\documentclass[a4paper,12pt,english]{amsart}

\usepackage{amssymb}

\input xypic
\xyoption{all}
\xyoption{arc}

\RequirePackage{calrsfs}
\DeclareSymbolFont{rsfscript}{OMS}{rsfs}{m}{b}
\DeclareSymbolFontAlphabet{\mathrsfs}{rsfscript}
\renewcommand{\mathcal}{\mathrsfs}

\def\CC{{\mathcal{C}}}
\def\EC{{\mathcal{E}}}
\def\JC{{\mathcal{J}}}
\def\XC{{\mathcal{X}}}

\def\RM{{\mathbb{R}}}
\def\NM{{\mathbb{N}}}
\def\ZM{{\mathbb{Z}}}

\def\SG{{\mathfrak{S}}}

\def\a{\alpha}
\def\b{\beta}
\def\g{\gamma}
\def\G{\Gamma}
\def\d{\delta}
\def\D{\Delta}

\def\ph{\varphi}
\def\l{\lambda}

\def\s{\sigma}

\def\th{\theta}

\def\to{\rightarrow}
\def\longto{\longrightarrow}

\def\fonctionl#1#2#3#4#5{\begin{array}{rccl}
{#1} : & {#2} & \longto & {#3} \\
& {#4} & \longmapsto & {#5} 
\end{array}}
\def\fonctio#1#2#3#4{\begin{array}{ccc}
{#1} & \longto & {#2} \\
{#3} & \longmapsto & {#4} 
\end{array}}

\def\lexp#1#2{\kern\scriptspace\vphantom{#2}^{#1}\kern-\scriptspace#2}
\def\le{\hspace{0.1em}\mathop{\leqslant}\nolimits\hspace{0.1em}}
\def\ge{\hspace{0.1em}\mathop{\geqslant}\nolimits\hspace{0.1em}}

\DeclareMathOperator{\supp}{{\mathrm{supp}}}
\DeclareMathOperator{\Ker}{{\mathrm{Ker}}}
\DeclareMathOperator{\restilde}{\widetilde{\mathrm{Res}}}

\DeclareMathOperator{\Stab}{{\mathrm{Stab}}}
\DeclareMathOperator{\Irr}{{\mathrm{Irr}}}
\DeclareMathOperator{\Ind}{{\mathrm{Ind}}}
\DeclareMathOperator{\Res}{{\mathrm{Res}}}

\def\itemth#1{\item[${\mathrm{(#1)}}$]}

\def\vide{\varnothing}
\def\DS{\displaystyle}
\def\SS{\scriptstyle}

\newcommand{\mpair}[1]{\pair{\,#1\,}}

\newcommand{\mset}[1]{\set{\,#1\,}}

\newcommand{\pair}[1]{\langle #1\rangle}

\newcommand{\set}[1]{\{#1\}}

\newcommand{\parmap}{j}

\newcommand{\Nat}{{\mathbb N}}
\newcommand{\real}{{\mathbb R}}

\newcommand{\mc}{\mathcal}
\newcommand{\ck}[1]{{#1}^\vee}
\newcommand{\QM}{{\mathbb  Q}}

\newtheorem{theorem}{Theorem}[section]

\newtheorem{proposition}[theorem]{Proposition}

\newtheorem{lemma}[theorem]{Lemma}

\newtheorem{corollary}[theorem]{Corollary}

\def\Wti{{\widetilde{W}}}
\def\Jti{{\widetilde{J}}}
\def\Xti{{\widetilde{X}}}
\def\AG{\mathfrak{A}}

\newcommand{\Piwt}{{\widetilde{\Pi}}}
\def\ellt{\tilde{\ell}}
\def\sti{{\tilde{s}}}
\def\tti{{\tilde{t}}}
\def\Tti{{\widetilde{T}}}
\def\Delt{{\widetilde{\D}}}
\def\AGt{{\widetilde{\AG}}}
\def\xti{{\tilde{x}}}
\def\mti{{\tilde{m}}}
\def\Mti{{\widetilde{M}}}
\def\Cti{{\widetilde{C}}}
\def\alpt{{\tilde{\alpha}}}
\def\bett{{\tilde{\beta}}}
\def\wti{{\tilde{w}}}
\def\TC{{\mathcal{T}}}

\def\dotcup{\hskip1mm\dot{\cup}\hskip1mm}

\begin{document}

\baselineskip=16pt

\title{Semidirect product decomposition \\
of Coxeter groups}

\author{C\'edric Bonnaf\'e}
\address{\noindent 
Laboratoire de Math\'ematiques de Besan\c{c}on (CNRS: UMR 6623), 
Universit\'e de Franche-Comt\'e, 16 Route de Gray, 25030 Besan\c{c}on
Cedex, France} 

\makeatletter
\email{cedric.bonnafe@univ-fcomte.fr}

\makeatother

\author{Matthew J. Dyer}
\address{Department of Mathematics\\
255 Hurley Building\\ University of Notre
Dame\\ Notre Dame, Indiana, 46556-4618\\U.S.A.}

\makeatletter
\email{dyer.1@nd.edu}

\thanks{The first author is partly supported by the ANR (Project 
No JC07-192339)}

\makeatother


\subjclass{According to the 2000 classification: 20F55}

\date{\today}

\begin{abstract} 
Let $(W,S)$ be a Coxeter system, let 
$S=I \hskip1mm{\cup}\hskip1mm J$ be a partition of $S$ 
such that no element of $I$ is conjugate to an element of $J$, 
let $\Jti$ be the set of $W_I$-conjugates of elements of $J$ 
and let $\Wti$ be the subgroup of $W$ generated by $\Jti$. 
We show that $W=\Wti \rtimes W_I$ and that 
$\Jti$ is the canonical set of Coxeter generators of the reflection subgroup $\Wti$
of $W$. We also provide algebraic and geometric conditions for an external semidirect product of Coxeter groups to arise in this way, and explicitly describe all such decompositions of (irreducible) finite Coxeter groups and affine Weyl groups. \end{abstract}

\maketitle

\pagestyle{myheadings}

\markboth{\sc C. Bonnaf\'e \& M. Dyer}{\sc Semidirect product 
of Coxeter groups}


\section*{Introduction}

\medskip

Let $(W,S)$ be a Coxeter system and assume that $S$ is the union
of two subsets $I $ and $J$ such that no element of $I $ is conjugate
to an element of $J$. Let $W_I$ be the subgroup of $W$ generated by $I$.
Let $\Jti$ be the set of elements of the form $wsw^{-1}$ where $w$ is in $W_I$
and $s$ is in $J$. Let $\Wti$ be the subgroup of $W$ generated by $\Jti$. 
In \cite{gal}, the following is shown:\bigskip

\noindent{\bf Theorem (Gal).}
{\it 
 With the above notation, we have:
\begin{itemize}
\itemth{a}  $W=  \Wti \rtimes W_I$ (semidirect product with $\Wti$ normal).

\itemth{b} $(\Wti,\Jti)$ is a Coxeter system 
\end{itemize}
}

\bigskip

\noindent{\sc Notation, remark, definition - } 
Let $T=\DS{\mathop{\cup}_{w \in W} wSw^{-1}}$ be the 
set of reflections of $W$. If $w \in W$, we set 
$N(w)=\{t \in T~|~\ell(wt) < \ell(w)\}$
where $\ell$ is the length function of $(W,S)$. 
If $W'$ is a subgroup of $W$ generated by reflections, we set 
$$\chi(W')=\{t \in T~|~N(t) \cap W'=\{t\}\}.$$
Then \cite[(3.3)]{dyer} $(W',\chi(W'))$ is a Coxeter system: $\chi(W')$ 
is called the set of {\it canonical Coxeter generators} of $W'$.

\bigskip

The following theorem clarifies the relation between the two natural  sets $\Jti$, $\chi(\Wti)$ of Coxeter generators
of the normal  reflection subgroup $\Wti$ of $W$.

\bigskip

\noindent{\bf Theorem.} 
{\it In the above setting, 
\begin{itemize} \itemth{a} $\Jti=\chi(\Wti)$ is the set of 
canonical Coxeter generators of $\Wti$.

\itemth{b} Each element $w$ of $W_I$ is the unique element 
of minimal length in its coset $\Wti w=w\Wti$.\end{itemize}}

\bigskip

We remark that our result that $\Jti=\chi(\Wti)$ is significantly 
stronger than the result of \cite{gal} that $(\Wti,\Jti)$ is a Coxeter system,
as it enables one to apply the  favorable combinatorial and algebraic
properties of the canonical Coxeter generators of  reflection subgroups
(see for example Lemma \ref{dyer X} and the remark after Corollary  
\ref{longueurs ajoutees}).

\medskip

In this paper, we provide a simple algebraic proof of both theorems above, independent of the
results of \cite{gal}, and  also describe algebraic   conditions (Theorem \ref{external}) under which an external semidirect product of Coxeter groups is naturally a Coxeter group.  We also  provide an alternative proof
(using root systems, see the proof of Theorem \ref{alternative}) of the fact that $\Jti=\chi(\Wti)$ 
 and of  a formula for the Coxeter matrix of $(\Wti, \Jti)$ obtained in \cite{gal}. 
Another main result is Theorem \ref{ext root}, which  is a  variant of Theorem  \ref{external}, providing geometric conditions for the external semidirect product of two Coxeter systems to be a Coxeter system, when each is attached to a  root system in the same  ambient real vector space and the Coxeter group attached to the first root system acts as a group of  automorphisms of the second based root system.  We include some general  results which are specific to the case in which  $W$ is finite or affine, including a construction of a homomorphism between Solomon descent algebras 
of $W$ and $\Wti$  when $W$ is finite. Finally, we describe explicitly by tables the  internal semidirect  product decompositions (as above)   of (irreducible)   finite Coxeter groups and affine Weyl groups.

\bigskip

\noindent{\sc Comment - } 
Semidirect product decompositions as above are used by the first author \cite{semicontinu} 
in  studying the Hecke algebra and the Kazhdan-Lusztig theory 
with unequal parameters whenever the parameters are zero on the 
set $I$. In fact, the authors were   unaware of \cite{gal} and originally   obtained both theorems above, together with  the same  description of the Coxeter matrix of $(\Wti,\Jti)$ as in \cite{gal}, by the above-mentioned root-system arguments. We are indebted to Koji Nuida
for bringing \cite{gal} to our attention.

\bigskip

\noindent{\sc Aknowledgements - } 
Part of this work was done while the first 
author was visiting the MSRI during the winter 2008. He wishes to 
thank the Institute for its hospitality and the organizers of the two 
programs for their invitation.

\newpage

\section{Internal semidirect products}

\medskip

For convenience, we restate the two Theorems of the introduction here in combined form.
\begin{theorem} \label{dyer} With the  notation of the introduction, we have:
\begin{itemize}
\itemth{a}  $W=  \Wti \rtimes W_I$ (semidirect product with $\Wti$ normal).

\itemth{b} $(\Wti,\Jti)$ is a Coxeter system. 

\itemth {c} $\Jti=\chi(\Wti)$ is the set of 
canonical Coxeter generators of $\Wti$.

\itemth{d} Each element $w$ of $W_I$ is the unique element 
of minimal length in its coset $\Wti w=w\Wti$.\end{itemize}
\end{theorem}
\begin{proof}
 If $s$ and $t$ are  elements of $T$, we denote by $m_{s,t}$ the order of $st$. 
It is well known that two simple reflections are $W$-conjugate iff, 
regarded as vertices of the Coxeter graph of $(W,S)$, 
there is a path from one to the other such that each edge of the path 
has either an odd label or no label (i.e. a label of $3$, which is omitted 
by the standard convention). In particular:
\begin{equation}\label{even}
\text{\it If $s \in I$ and $t \in J$, then $m_{s,t}$ is even.}
\end{equation}

We first prove (a). Let 
$$\fonctionl{\ph}{S}{W_I}{s}{\begin{cases}
s & \text{if $s \in I$,}\\
1 & \text{otherwise.}
\end{cases}}$$
It follows easily from (\ref{even}) that 
$(\ph(s)\ph(t))^{m_{st}}=1$ for all $s$, $t \in S$. Therefore, 
there exists a unique morphism of groups $W \to W_I$ extending $\ph$:
we still denote it by $\ph$. Since $\ph(w)=w$ for all $w \in W_I$, 
it is sufficient to prove that 
\begin{equation}\label{kernel}
\Ker \ph = \Wti.
\end{equation}
Let us prove (\ref{kernel}). First of all, note that $\Wti \subseteq \Ker \ph$. 
So it is enough to show that $W=\Wti W_I$. For this, it is sufficient to show that if $w\in W\setminus W_I$, there is some $t\in \Jti$ with
$l(tw)<l(w)$. Write $w=s_1\cdots s_n$ (reduced) with all $s_i \in S$.
Since $w\not\in W_I$, there is some $j$ with $s_j \in J$. Without 
loss of generality, assume that $j$ is minimal with this property. Then 
$t:=s_1\cdots  s_{j-1}s_j s_{j-1}\cdots s_1\in \Jti$ and 
$l(tw)<l(w)$ as required. This completes the proof of (a).

We claim next  that  $\Wti\cap T$  consists of all $W$-conjugates of elements of $J$.
In fact,  since $\Wti$ is generated by $\Jti\subseteq T$,
\cite[3.11(ii)]{dyer} implies that $\Wti\cap T$ consists of the $\Wti$-conjugates of elements of $\Jti$; since $\Jti$ consists of the  $W_{I}$-conjugates of elements of $J$, (a) implies that 
the $\Wti$-conjugates of elements of $\Jti$ are exactly the  $W$-conjugates of elements of $J$, completing the proof of the claim. 

We can  now prove (c) (which immediately implies (b)). 
Regard the power set $\mathcal{P}(T)$
of $T$ as an abelian group under symmetric difference $A+B:=(A\cup B)\setminus (A \cap B)$
and with natural $W$-action $(w,A)\mapsto wAw^{-1}:=\mset{wtw^{-1}\mid t\in A}$.
The function $N\colon W\rightarrow \mathcal{P}(T)$ is characterized by the cocycle condition  $N(xy)=N(x)+xN(y)x^{{-1}}$ for $x,y\in W$, and its special values $N(s)=\set{s}$ for $s\in S$
(see \cite{dyer}). Consider an element $t\in \Jti$, say $t=wrw^{{-1}}$ where $w\in W_{I}$ and $r\in J$. We have 
\[ N(t)=N(wrw^{{-1}})=N(w)+wN(r)w^{{-1}}+wrN(w^{-1})rw^{{-1}}\]
Note that  for $x\in W_{I}$, $N(x)\subseteq W_{I}\cap T$ consists of reflections  which are $W$-conjugate to elements of $I$, and hence all elements
of $N(w)\cup wrN(w^{{-1}})rw^{{-1}}$ are $W$-conjugate to elements of $I$. 
From the claim and the assumption that no element of $I$ is $W$-conjugate 
to an element of $J$, it  therefore follows that   
$N(t)\cap \Wti=wN(r)w^{{-1}}\cap \Wti=\set{t}$ i.e. $t\in \chi(\Wti)$.
This proves that $\Jti\subseteq \chi(\Wti)$. Now $\chi(\Wti)$ is a set of Coxeter generators of
$\Wti$,  and hence it is a minimal (under inclusion) set of generators of  $\Wti$ by 
\cite[Ch IV, \S 1,  Cor  3]{bourbaki}. Since $\Jti$ generates $\Wti$ by definition, we get
(b)--(c)

We now  state a Lemma
needed in the proof of   (d) (and elsewhere  in this paper).

\begin{quotation}
{\small 
\begin{lemma}\label{dyer X}
Let $W'$ be a subgroup of $W$ generated by reflections, let $S'=\chi(W')$ and 
let $X'$ be the set of elements $x \in W$ such that $x$ has minimal length 
in $xW'$. Then:
\begin{itemize}
\itemth{a} Every coset in $W/W'$ contains a unique element of $X'$. 

\itemth{b} An element $x \in W$ belongs to $X'$ if and only if 
$\ell(xt) > \ell(x)$ for all $t \in S'$. 

\itemth{c}  If $x\in X'$, $w\in W'$ and $t\in W'\cap T$, then $\ell(xwt)>\ell(xw)$ iff
$\ell(wt)>\ell(w)$ iff $\ell'(wt)>\ell'(wt)$ where $\ell'$ is the length function of $(W',S')$.
\end{itemize}
\end{lemma}

\begin{proof}
See \cite[(3.4)]{dyer}.
\end{proof}
}
\end{quotation}

If $t\in \Jti$ and $w\in W_I$, then $l(wt)>l(w)$ since 
$t\not\in W_I$. Since $\Jti$ is the set of canonical generators of $\Wti$, 
this implies that $w\in W_I$ is the (unique) element of minimal length in 
its  coset $w\Wti $ by Lemma \ref{dyer X} (a)--(b). Theorem \ref{dyer} (d) follows from 
 the fact that $W=W_{I}\Wti$. \end{proof}
\medskip

In the remainder of this section, we give some simple complements to and  consequences of Theorem
\ref{dyer}, with applications  in \cite{semicontinu} or  to explicit  computation in examples, and then describe the Coxeter matrix of $(\Wti, \Jti)$.

Let $\ellt\colon \Wti\rightarrow \Nat$ denote the length function of $(\Wti,\Jti)$. If $w \in W$, we denote by $\ell_I(w)$ (respectively 
$\ell_J(w)$) the number of occurrences of elements of $I$ 
(respectively $J$) in a reduced expression of $w$ (note that 
these two numbers do not depend on the choice of the reduced 
expression and that $\ell(w)=\ell_I(w)+\ell_J(w)$).

\medskip

\begin{corollary}\label{longueurs ajoutees}
Let $a \in W_I$ and $w \in \Wti$. Then 
$\ell_J(aw)=\ell_J(wa)=\ell_J(w)=\ellt(w)$ and $\ellt(awa^{-1})=\ellt(w)$. 
\end{corollary}


\begin{proof}
Since $W_I$ acts on $\Wti$ by preserving $\Jti$, the length function 
$\ellt$ is invariant by $W_I$-conjugation, so $\ellt(awa^{-1})=\ellt(w)$. 
Also, if $s \in I$ and $x \in W$, we have $\ell_J(sx)=\ell_J(x)=\ell_J(xs)$, 
so this shows that $\ell_J(aw)=\ell_J(wa)=\ell_J(w)$. 

It remains to show that $\ell_J(w)=\ellt(w)$. 
We argue by induction on $\ellt(w)$. The result is clear if $\ellt(w)=0$. 
So assume that $\ellt(w) > 0$. Then there exists $\tti \in \Jti$ such that 
$\ellt(\tti w) < \ellt(w)$. Let $x \in W_I$ and $t \in J$ be such that 
$\tti=xtx^{-1}$. Let $w'=x^{-1}wx$. Since $\ellt$ is invariant 
by $W_I$-conjugation, we get that $\ellt(w)=\ellt(w')$ and 
$\ellt(\tti w)=\ellt(t w')$. Therefore $\ellt(tw')=\ellt(w')-1$ 
and so, by Lemma \ref{dyer X} (c), we get that $\ell(tw')=\ell(w')-1$. 
In particular, $\ell_J(tw')=\ell_J(w')-1$. So, by the induction 
hypothesis, $\ell_J(w')-1=\ell_J(tw')=\ellt(tw')=\ellt(w')-1$. 
Therefore, $\ell_J(w')=\ellt(w')$ and 
so $\ell_J(w)=\ell_J(w')=\ellt(w')=\ellt(w)$. 
\end{proof}

\noindent{\sc Remark -}  
We observe that  the Corollary is not an obvious consequence of the results proved in \cite{gal};
for example, the proof above requires Theorem \ref{dyer}(c), and not just Theorem \ref{dyer}(a)--(b).
\bigskip

The next results  require  some additional  notation. 
 If $\tti \in \Jti$ and if $t$ and $t' \in J$ and $x$, $x' \in W_I$ 
are such that $\tti=xtx^{-1}=x't'x^{\prime -1}$, then 
\begin{equation}\label{equal}
t=t'.
\end{equation}
Indeed, in this case, then $t' \in W_{I \cup \{t\}} \cap W_J=<t>$. 
Therefore, if $\tti \in \Jti$, we can define $\nu(\tti)$ as 
the unique element of $J$ which is conjugate to $\tti$ under $W_I$. 

\begin{corollary}\label{conj}
If $\tti$, $\tti' \in \Jti$ are $\Wti$-conjugate, then 
$\nu(\tti)$ and $\nu(\tti')$ are $W$-conjugate.
\end{corollary}

\bigskip

\noindent{\sc Remark -}  
Recall that an isomorphism of Coxeter systems 
$(W_1,S_1)\rightarrow (W_2,S_2)$ is a group isomorphism 
$W_1\rightarrow W_2$ inducing a bijection $S_1\rightarrow S_2$.
In  the semidirect product decomposition $W=\Wti \rtimes W_I$ of Theorem 1,
it is clear that the induced action by conjugation of $W_I$ on $\Wti $ is by 
automorphisms of the Coxeter system $(\Wti ,\Jti)$. Moreover, the set 
of Coxeter generators $S$ of $W$ is the disjoint union of the set $I$ 
of Coxeter generators of $W_I$ and the set $J$ of $W_I$-orbit 
representatives on $\Jti$.

 \bigskip

In order to parametrize $\Jti$, we must first determine the centralizer 
of $t \in J$ in $W_I$.  If $s \in S$, we set 
$s^\perp=\{r\in S~|~sr=rs\}$.

\medskip

\begin{lemma}\label{centralisateur}
Let $t \in J$. Then $C_{W_I}(t)=W_{I \cap t^\perp}$.
\end{lemma}


\begin{proof}
First it is clear that $W_{I \cap t^\perp} \subseteq C_{W_I}(t)$. Conversely, 
let $w \in W_I$ be such that $wt=tw$. Let $w=s_1\cdots s_r$ be a reduced 
expression of $w$ (so that $s_i \in I$). Then, $s_1\cdots s_r t$ and 
$ts_1\cdots s_r$ are reduced expression of the same element $wt=tw$ of $W$. 
By Matsumoto's lemma, this means that one can obtain one of these 
reduced expression by applying only braid relations. But $t$ occurs only once 
in both reduced expressions: this means that, in order to make $t$ pass from the first 
position to the last position, $t$ must commute with all the $s_i$. So 
$w \in W_{I \cap t^\perp}$. 
\end{proof}

\bigskip

We set
$$\JC=\{(x,t)~|~t \in J\text{~and~}x \in X_{I \cap t^\perp}^I\}.$$
Then it follows from \eqref{equal} and Lemma \ref{centralisateur} 
that the map
\begin{equation} \label{j tilde}
\parmap\colon \fonctio{\JC}{\Jti}{(x,t)}{xtx^{-1}}
\end{equation}
is bijective. 

\bigskip

\begin{proposition}\label{red exp}
Let $(x,t) \in \JC$. Then:
\begin{itemize}
\itemth{a} For $w\in W_I$, one has $w\parmap(x,t)w^{-1}=\parmap(x',t)$ 
where $x'$ is the unique element of $X_{I\cap t^\perp}^I$ with 
$x'W_{I\cap t^\perp}=wxW_{I\cap t^\perp}$.

\itemth{b} The palindromic reduced expressions of $xtx^{-1}$ in $(W,S)$ 
are precisely the expressions $t_n\cdots t_1t_0t_1\cdots t_n$  
such that $t_n\cdots t_1$ is a reduced expression for $x$ in $(W_I,I)$ 
and $t_0=t$.
\end{itemize}
\end{proposition}


\begin{proof} 
Part (a) is immediate from the definitions. 
For (b), we first recall the following result:

\begin{quotation}
\small{
\begin{lemma}\label{palind}
If $r_1\cdots  r_{2m+1}$ is a reduced expression for a reflection $t\in T$, 
then  $r_1\cdots  r_mr_{m+1}r_m\cdots r_1$ is a palindromic reduced 
expression of $t$.
\end{lemma}

\begin{proof} 
See \cite[(2.7)]{dyer}.
\end{proof}}
\end{quotation}

Write $l(xtx^{-1})=2m+1$. 
We have $xtx^{-1}\in W_{I\cup\set{t}}$, so any reduced expression 
$xtx^{-1}=s_1\cdots s_{2m+1}$ for $xtx^{-1}$ has all $s_i\in I\cup\{t\}$. 
Note $s_1\cdots s_m s_{m+1}s_m\cdots  s_1$ is also a reduced expression for 
$xtx^{-1}$ by Lemma \ref{palind}. Thus,  $t\in J$ is $W$-conjugate to 
$s_{m+1}\in I\cup\set{t}$ and so $s_{m+1}=t$.  
Let $t_n\cdots t_1$ be a reduced expression for $x$, and $t_0=t$. 
Then   $xtx^{-1}=t_n\cdots t_1t_0t_1\cdots t_n$  and the right hand side 
contains some  reduced expression $s_1\cdots  s_{2m+1}$ for $xtx^{-1}$ 
as a subexpression. By the above, we have $s_{m+1} =t=t_0$, which is the 
only occurrence of $t$ in $t_n\cdots t_0\cdots t_n$. Hence 
$s_1\cdots s_m s_{m+1}s_m\cdots  s_1$ is also a  reduced expression 
for $xtx^{-1}$ contained as a subexpression  of $t_n\cdots t_0\cdots t_1$. 
Let $y=s_1\cdots  s_m\in W_I$.  Then $xtx^{-1}= yty^{-1}$ so 
$z:=y^{-1}x\in C_{W_I}(t)=W_{I\cap t^\perp}$. We have $y=xz^{-1}$ 
with $\ell(y)=m=\ell(xz^{-1})=\ell(x)+\ell(z^{-1})=n+\ell(z^{-1})$ and  
$m\le n$, so $m=n$. This shows $t_n\cdots t_0\cdots t_n$ is a 
reduced expression for $xtx^{-1}$.

\medskip

Since every reduced expression for $xtx^{-1}$ has $t$ as its middle 
element,  it follows that this central $t$ can  never be involved 
in a braid  move between reduced expressions for $xtx^{-1}$, and the  
conclusion of (b) is clear.
\end{proof}

\bigskip

Now we introduce notation to describe the Coxeter matrix of $(\Wti,\Jti)$. If $A$, $B$ and $C$ are three subsets 
of $S$ such that $B \subseteq A$ and $C \subseteq A$, we denote by 
$X_{BC}^A$ the set of $x \in W_{A}$ which have minimal length in 
$W_B x W_C$. For simplicity, we set $X_{\vide C}^A=X_C^A$.  We shall use
Deodhar's Lemma \cite[Lemma 2.1.2]{geck}, which  amounts to to the statement  that if $w\in W_{C}^{A}$
and $s\in A$ with $sw\not\in W_{C}^{A}$ then
$\ell(sw)>\ell(w)$ and $sw=wr$ for some $r\in C$.

\medskip

Now, let $\sti$, $\tti \in \Jti$ and let $s=\nu(\sti)$ and $t=\nu(\tti)$. 
Then there exists $x$ and $y \in W_I$ such that $\sti=xsx^{-1}$ and 
$\tti=yty^{-1}$. We denote by $f(\sti,\tti)$ the unique element of 
$X_{I \cap s^\perp,I \cap t^\perp}^I$ such that 
$x^{-1}y \in W_{I \cap s^\perp} f(\sti,\tti) W_{I \cap t^\perp}$. 
It is readily seen that $f(\sti,\tti)$ depends only on $\sti$ and $\tti$ 
and not on the choice of $x$ and $y$. Note that $\sti=\tti$ if and only 
if $s=t$ and $f(\sti,\tti)=1$. Recall that if $s$ and $t$ are two elements of $T$, $m_{s,t}$ denotes the order of $st$. We then set:
$$\mti_{\sti,\tti}=\begin{cases}
1 & \text{if $\sti=\tti$,}\\
m_{s,u}/2 & \text{if $s=t$ and $f(\sti,\tti)=u \in I$,}\\
\infty & \text{if $s=t$ and $\ell(f(\sti,\tti)) \ge 2$,}\\
m_{s,t} & \text{if $s \neq t$ and $f(\sti,\tti)=1$,}\\
\infty & \text{if $s \neq t$ and $f(\sti,\tti)\neq 1$.}\\
\end{cases}$$
We denote by $\Mti$ the matrix $(\mti_{\sti,\tti})_{\sti,\tti \in \Jti}$.

Since $f(\sti,\tti)=f(\tti,\sti)^{-1}$ and since $m_{su}$ is even 
if $s \in J$ and $u \in I$ by (\ref{even}), we have, for all $\sti$, 
$\tti \in \Jti$ and $x \in W_I$,
\begin{equation}\label{symetrie}
\begin{cases}
\mti_{\sti,\tti} \in \ZM_{\geqslant 1},~&\\
\mti_{\sti,\tti}=\mti_{\tti,\sti},~&\\
\mti_{x\sti x^{-1},x\tti x^{-1}}=\mti_{\sti,\tti},~\\ 
\mti_{\sti,\sti}=1,&\\
\mti_{\sti,\tti} \geqslant 2 ~(\text{if $\sti \neq \tti$}).\\
\end{cases}
\end{equation}
The last inequality follows from the fact that, if $f(\sti,\tti)=u \in I$, 
then $us\neq su$.

\bigskip
The following is proved by a simple algebraic  argument,
 independent of the proof of Theorem \ref{dyer}, in \cite{gal}.
 For a different proof using root systems, see Theorem \ref{alternative}.
\medskip
\begin{theorem}[Gal]\label{gal2}
For $\sti,\tti\in \Jti$, the product $\sti\tti$ in $W$ has order $\mti_{\sti,\tti}$ i.e.
the matrix $\Mti$ defined above is   the Coxeter matrix of $(\Wti,\Jti)$.
\end{theorem}

\bigskip

\begin{corollary}\label{j}
Let $L$ be a subset of $\Jti$ such that the Coxeter graph 
of $(\Wti_L,L)$ is an irreducible component of the one of $(\Wti,\Jti)$. 
Then $J=\{\nu(\tti)~|~\tti \in L\}$. 
\end{corollary}

\begin{proof}
Let $K=\{\nu(\tti)~|~\tti \in L\} \subseteq J$. Then $K$ is not empty. 
We shall prove by induction on $n$ the following assertion 
(from which the Corollary follows easily):
\begin{quotation}
\begin{itemize}
\item[$(P_n)$] {\it Let $s \in K$ and let $t \in J$ be such that 
there exists a path of length $n$ from $s$ to $t$ in the Coxeter graph 
of $(W,S)$. Then $t \in K$.}
\end{itemize}
\end{quotation}

It is clear that $(P_0)$ holds. Let us show $(P_1)$. So let $s \in K$ 
and $t \in J$ be such that $m_{s,t} \ge 3$. Then there exists 
$x \in W_I$ such that $xsx^{-1} \in L$. Since 
$m_{xsx^{-1},xtx^{-1}} = m_{s,t} \ge 3$ and $xtx^{-1} \in \Jti$, we get 
that $xtx^{-1} \in L$ and so $t=\nu(xtx^{-1}) \in K$, as expected. 

Now let $n \ge 2$ and assume that 
$(P_0)$, $(P_1)$,\dots, $(P_{n-1})$ hold. Let $s \in K$ and let $t \in J$ 
be such that there exists a path of length $n$ from $s$ to $t$ 
in the Coxeter graph of $(W,S)$. 
Set $s_0=s$ and $s_n=t$ and 
let $s_1$,\dots, $s_{n-1}$ be elements of $S$ such that 
$m_{s_{i-1},s_i} \ge 3$ for all $i \in \{1,2,\dots,n\}$. 
We may assume that $s_{i} \neq s_{j}$ if $i\neq j$. 
If there exists $i \in \{1,2,\dots,n-1\}$ such that $s_i \in J$, 
then the induction hypothesis (applied twice) 
implies that $s_i \in K$ and that $s_n=t \in K$. 
So we may assume that $s_i \in I$ for all $i \in \{1,2,\dots,n-1\}$. 

Let $x \in W_I$ be such that $\sti=xsx^{-1} \in L$. 
Let $y = s_1 s_2 \cdots s_{n-1} \in W_I$ and let $\tti=xyty^{-1}x^{-1}$. 
Note that $\ell(y)=n-1$. 
Since $s \neq t$, we get that 
$m_{\sti,\tti}=\infty$ if 
$y \not\in W_{I\cap s^\perp} \cdot W_{I \cap t^\perp}$. So it remains 
to show that $y \not\in W_{I\cap s^\perp} \cdot W_{I \cap t^\perp}$. 
But, if $y \in W_{I\cap s^\perp} \cdot W_{I \cap t^\perp}$, this would 
imply that $y$ has a reduced expression of the form 
$y=\s_1\s_2\cdots \s_{n-1}$ such that there exists $k \in \{0,1,2,\dots,n-1\}$ 
satisfying $\s_1$,\dots, $\s_k \in I \cap s^\perp$ and 
$\s_{k+1},\dots,\s_{n-1} \in I \cap t^\perp$. 
Since $y$ has only one reduced expression, this means 
that $s_1 \in I \cap s^\perp$ or $s_{n-1} \in I \cap t^\perp$. 
This is impossible, and so the proof of $(P_n)$ is complete.
\end{proof}

\begin{corollary}\label{transitif}
Assume that $(W,S)$ is irreducible. Then $W_I$ permutes transitively 
the irreducible components of $(\Wti,\Jti)$.
\end{corollary}

\begin{proof}
Let $L$ and $L'$ be two subsets of $\Jti$ such that the Coxeter 
graphs of $(\Wti_L,L)$ and $(\Wti_{L'},L')$ are irreducible 
components of $(\Wti,\Jti)$. Let $s \in J$. By Corollary \ref{j}, 
there exist $x$ and $y$ in $W_I$ such that $xsx^{-1} \in L$ 
and $ysy^{-1} \in L'$. So $L \cap \lexp{xy^{-1}}{L'} \neq \vide$. 
Since $W_I$ permutes the irreducible components of the Coxeter 
graph of $(\Wti,\Jti)$, we get that $L=\lexp{xy^{-1}}{L'}$.
\end{proof}

\bigskip

\noindent{\bf Parabolic subgroups, cosets.} 
We close this section by investigating the relationships between 
standard parabolic subgroups of $W$ and $\Wti$, as well as between 
the sets of distinguished cosets representatives. Roughly speaking, 
with respect to these questions, $\Wti$ behaves like a standard 
parabolic subgroup. 

If $L$ is a subset of $\Jti$, we note by $\Wti_L$ the subgroup of $\Wti$ 
generated by $L$. If $K$ is a subset of $S$, we set
$$K^+=\{wtw^{-1}~|~w \in W_{I \cap K}\text{ and }t \in J \cap K\}.$$
It is a subset of $\Jti$. Then

\begin{proposition}\label{inter parabolique}
Let $K$ be a subset of $S$. Then $K^+=W_K \cap \Jti$ and 
$W_K \cap \Wti = \Wti_{K^+}$.
\end{proposition}

\begin{proof}
Let $\ph_K : W_K \to W$ denote the restriction to $W_K$ of the morphism 
$\ph : W \to W_I$ defined in the proof of Theorem \ref{dyer}. Then 
$\Wti_{K^+} \subseteq \Ker \ph_K$ and, if $w \in W_{I \cap K}$, we have 
$\ph_K(w)=w$. But, by Theorem \ref{dyer} applied to $W_K$ and 
to the partition $K=(I \cap K) \dotcup (J \cap K)$, we have 
$W_K=W_{I \cap K} \ltimes \Wti_{K^+}$. Therefore, $\Wti_{K^+}=\Ker \ph_K$. 
But, by \ref{kernel}, $\Ker \ph_K = \Wti \cap W_K$. This shows 
the second equality of the proposition. The first one then 
follows easily.
\end{proof}

If $L$ is a subset of $\Jti$, we denote by $\Xti_L$ (respectively 
$X_L$) the set of elements $w$ of $\Wti$ (respectively $W$) which 
have minimal length in $w \Wti_L$. 

\begin{lemma}\label{produit x}
Let $L$ be a subset of $\Jti$. Then the map 
$$\fonctio{W_I \times \Xti_L}{X_L}{(w,x)}{wx}$$
is bijective.
\end{lemma}

\begin{proof}
First, it follows from Theorem \ref{dyer} (a) that 
the map 
$$\fonctio{W_I \times \Xti_L}{W/\Wti_L}{(w,x)}{wx\Wti_L}$$
is bijective. So it remains to show that, if $w \in W_I$ and 
$x \in \Xti_L$, then $wx \in X_L$. But this follows from 
Lemma \ref{dyer X} (c).
\end{proof}

We conclude by an easy result on double coset representatives:

\begin{proposition}\label{double}
Let $K$ be a subset of $S$. Then the map 
$$\fonctio{X_{I \cap K}^I}{\Wti\backslash W / W_K}{d}{\Wti d W_K}$$
is bijective. Moreover, if $d \in X_{I \cap K}^I$, then 
$d$ is the unique element of minimal length in $\Wti d W_K=d \Wti W_K$ and 
$$\Wti \cap \lexp{d}{W_K}=\Wti_{\Jti \cap \lexp{d}{W_K}}.$$
\end{proposition}

\begin{proof}
First, $\Wti\backslash W / W_K=W/\Wti W_K=W/W_K\Wti$ and, since 
$W_K = W_{I \cap K} \ltimes \Wti_{K^+}$, we have 
$$W/W_K\Wti=W/(W_{I \cap K} \ltimes \Wti) \simeq W_I/W_{I \cap K}.$$
This shows the first assertion. 

Now, let $d \in X_{I \cap K}^I$. Then, since $\Wti$ is normal in $W$, 
we get 
$$\Wti \cap \lexp{d}{W_K}=\lexp{d}{(\Wti \cap W_K)}=\lexp{d}{\Wti_{K^+}}$$
(see Proposition \ref{inter parabolique}). But $W_I$ acts on the 
pair $(\Wti,\Jti)$, so 
$$\lexp{d}{\Wti_{K^+}}=\Wti_{\lexp{d}{K^+}}.$$
Now, by Proposition \ref{inter parabolique}, we have 
$$\lexp{d}{K^+}=\lexp{d}{(\Jti \cap W_K)}=\Jti \cap \lexp{d}{W_K}.$$
So the last assertion follows.

It remains to show that $d$ is the unique element of minimal length 
in $\Wti d W_K$. We have $\Wti d W_K= d(W_{I \cap K} \ltimes \Wti)$. 
Let $x \in W_{I \cap K}$ and $w \in \Wti$ be such that 
$\ell(dxw) \le \ell(d)$. Then, by Theorem \ref{dyer} (d), 
$\ell(dxw) \ge \ell(dx) =\ell(d)+\ell(x)$, so $x=1$. 
Again by Theorem \ref{dyer} (d), we get $\ell(dw) =\ell(d)$, 
so $w=1$, as expected.
\end{proof}

\bigskip

\section{External semidirect products}

\medskip

In this section, we  discuss the converse of Theorem \ref{dyer}(a)--(b), giving  conditions 
which imply that an external  semidirect product of Coxeter groups is a 
Coxeter group.

\medskip

Let  $(W',I)$ and $(\Wti ,\Jti)$ be Coxeter systems and 
$\theta\colon W'\rightarrow \text{\rm Aut}(\Wti ,\Jti)$ 
be a group homomorphism, where the right hand side is the group 
of automorphisms of $(\Wti,\Jti)$.  One may regard $\theta$ as a 
homomorphism from $W'$ to the automorphism group of $\Wti$, and form 
the semidirect product of groups $W:=\Wti\rtimes W'$, with $\Wti $ normal. 
We regard $W'$ and $\Wti $ as subgroups of $W$ in the usual way. Thus, 
every element  $w$ of $W$ has a unique expression $w=\widetilde ww'$ 
with $w'\in W'$ and $\widetilde w\in\Wti $. The product in $W$ is determined 
by the equation $w'\widetilde{w}w^{\prime -1}=\theta(w')( \widetilde w)$ 
for $w'\in W'$, $\widetilde w\in \Wti $.    

\bigskip

\begin{theorem}\label{external}  
Fix a set $J$ of $W_I$-orbit representatives on $\Jti$, 
and set $S:=I{\cup} J$. 
For any $s\in S$, let $s^\perp:=\mset{r\in S\mid rs=sr}$.
Then $(W,S)$ is a Coxeter system iff the conditions $(1)$ 
and $(2)$ below hold:
\begin{itemize}
\itemth{1} for all $r,s\in J$ and $u\in W'$ with $r=usu^{-1}$, 
one has  $r=s$  and $u\in W_{\! I\cap r^\perp}'$. 
  
\itemth{2} for all $r\in J$ and $s\in \Jti$ with $r\neq s$ and 
$rs$ of finite order,  either $\text{\rm (i)}$ or $\text{\rm (ii)}$ 
below holds:
\begin{itemize}
\itemth{i} $s=utu^{-1}$ for some  $u\in W_{\! I \cap r^\perp}'$ and 
$t\in J$ with  $t\neq r$ and  $rt$ of finite order  

\itemth{ii} $s=uvrvu^{-1}$ for some $u\in W_{\! I \cap r^\perp}'$ and  
$v\in I$ with $rv$ of finite order greater than $2$.
\end{itemize} 
\end{itemize}
\end{theorem}


\begin{proof} 
It is easy to see that $S$ is a set of involutions generating $W$.  
No element of $I$ is $W$-conjugate to an element of $J$
(since any $W$-conjugate of an element of $J$ is in $\Wti $);
in particular, the union $S=I\dot \cup J$ is disjoint
(we shall use  $\dot\cup$ to  denote disjoint union throughout this paper). Moreover, 
a simple computation shows that for $s\in I$ and $r\in J$, the order of $sr$ 
in $W$ is even, equal to twice the order of $r'r$ in $\Wti $ 
where $r'=\theta(s)(r)=srs$.
 
\medskip

For $r,s\in S$, let $m_{r,s}$ denote the order of $rs$. We have $m_{r,r}=1$ 
and $m_{r,s}=m_{s,r}\in \Nat_{\geq 2}\cup\set{\infty}$ for all $r\neq s$. 
Let $(U,S)$ be a Coxeter system with Coxeter matrix $m_{r,s}$ i.e. 
$U$ is a Coxeter group with $S$ as its set of Coxeter generators, and the 
order of $rs$ in $U$ is $m_{r,s}$ for all $r,s\in S$.

\medskip

For any  $K\subseteq S$,  let $U_K$ denote the standard parabolic  subgroup 
of $U$ generated by $K$.  Let $\Jti'$ denote the subset of $U$  consisting 
of all products $usu^{-1}$ in $U$ with $s\in J$ and $u\in U_I$, and 
let $\widetilde {U}$ denote the subgroup of $U$ generated by $\Jti'$.
No element of $I$ is conjugate in $U$ to an element of $J$, since $m_{r,s}$ 
is even for all $r\in I$ and $s\in J$. Hence, by Theorem \ref{dyer}, there 
is a semidirect product decomposition $U=U_I\ltimes \widetilde{U} $ 
with $\widetilde{U}$ normal in $U$.

\medskip

Since  $rs$ has the same order $m_{r,s}$ in both $U$ and $W$, for any  
$r,s\in S$, there is a group epimorphism $\pi\colon U\rightarrow W$  which 
is the identity on $S$. The homomorphism $\pi$ restricts to an isomorphism 
of Coxeter systems $(U_I,I)\rightarrow (W',I)$ (which we 
henceforward regard as an identification) and  $\pi$ also restricts to 
an isomorphism of Coxeter systems $(U_J,J)\rightarrow (\Wti_J,J)$. 
Further, $\pi$ restricts to a surjective, $W'$-equivariant
(for the conjugation actions by $W'$)  group homomorphism 
$\tilde \pi\colon \widetilde{U}\rightarrow \Wti$  and  
$\tilde \pi$ restricts further to a surjective map of $W'$-sets 
$\pi'\colon \Jti'\rightarrow \Jti$. 

\medskip

Now if $(W,S)$ is a Coxeter system, the validity of the conditions of  Theorem \ref{external} (1)--(2) follows readily from \eqref{j tilde} and Theorem \ref{gal2}. (In this case, the 
map $\tilde \pi$ is of course an isomorphism of Coxeter systems).

\medskip

Conversely, suppose that (1) and (2) hold. It will suffice to show that 
$\tilde \pi$ is an isomorphism of Coxeter systems. First, we show that 
$\pi'$ is injective. Consider two arbitrary elements $uru^{-1}$ and 
$vsv^{-1}$ of $\Jti'$, with $u,v\in W'$ and $r,s\in  J$. Assume 
$\pi(uru^{-1})=\pi(vsv^{-1})$ i.e. $u\pi(r)u^{-1}=v\pi(s)v^{-1}$.
Then $\pi(r)=x\pi(s)x^{-1}$ where $x=u^{-1}v\in W'$. By (1), $r=s$ and 
$x\in W'_{I\cap r^\perp}$. By the defining relations for $(U,S)$, 
it follows that $r=xsx^{-1}$ in $U$, so $uru^{-1}=vsv^{-1}$ in $U$.
Hence $\pi'$ is injective, and in fact bijective since 
we noted above that $\pi'$ is a surjection.

\medskip

Now it will suffice to show that for all distinct  $r',s'\in \Jti'$,   
$r's'$ has the same order in $U$ as $\pi(r')\pi(s')$ has in $W$. Using the 
$W'$-equivariance of $\tilde  \pi$, we may assume that $r'=r\in J$ and 
$s'=s\in \Jti'$. Also, we may assume that $\pi(r)\pi(s)$ has finite order 
$n>1$ in $W$, without loss of generality. We have by (2) that 
either  $\pi(s)=u \pi(t)u^{-1}$ for $u\in W_{\! I\cap r^\perp}'$, 
$t\in J$ with $t\neq r$ and $\pi(t)\pi(r)$ of finite order, or 
$\pi(s)=uv\pi(r)vu^{-1}$ for some $u\in W_{\! I\cap r^\perp}'$ and  
$v\in I$ with $v\pi(r)$ of finite order greater than $2$.
In the first  (resp., second) case, $\pi(r)\pi(s) =u\pi(r)\pi(t)u^{-1}$
(resp.,  $\pi(r)\pi(s)=u\pi(r)v\pi(r)vu^{-1}$) and  $n$ is the order of 
$\pi(r)\pi(t)$  (resp., half the order of $\pi(r)v)$ in $W$. In the first  
case, $s=utu^{-1}$. The relations of $(U,S)$ imply that  $rs=urtu^{-1}$, 
which  has the same order  as $rt$ in $U$. In the second case,
$s=uvrvu^{-1}$ and the relations of $(U,S)$ imply that $rs=urvrvu^{-1}$, 
which has order equal to half the order of $rv$ in $U$. The definition of 
$U$ implies that the order of $rt$ (resp., $rv$) in $U$ is the same as that 
of $\pi(r)\pi(t)$ (resp., $\pi(r)v$) in $W$ and so   the order of $rs$ in 
$U$ is equal to the order  $n$ of $\pi(r)\pi(s)$ in $W$ in either case, 
completing the proof. 
\end{proof}

\medskip

\noindent{\sc Remark - } 
We leave open the question of  whether different choices of the set 
$J$ of $W_I$-orbit representatives satisfying the conditions 
in Theorem \ref{external} are possible, or if possible, would give rise to isomorphic Coxeter 
systems  $(W,I\cup J)$. 

\bigskip

\section{Semi-direct products and root systems}

\medskip

 In this section, we use the  standard geometric 
realization of $(W,S)$ as a reflection group associated to a based 
root system. In fact, it is convenient (and essential for the main result Theorem \ref{ext root} of this section)
to introduce a slightly more general 
class of geometric realizations with better ``functoriality'' properties with respect to  inclusions of reflection subgroups.

\medskip

Let $\EC$ be a $\RM$-vector space equipped with a symmetric $\RM$-bilinear form 
$\langle,\rangle$.  We say a  subset $\Pi$ of $\EC$ is positively independent
if $\sum_{\alpha\in \Pi}c_\alpha \alpha=0$ with all $c_\alpha\geq 0$ implies that all $c_\alpha=0$.
For example, any $\RM$-linearly independent set is positively independent.
If $\a \in \EC$ is such that $\mpair{\alpha,\alpha}=1$, 
we set $\a^\vee=2\a$ and we define 
$$\fonctionl{s_\a}{\EC}{\EC}{v}{v-\langle v, \a^\vee \rangle \a.}$$
Then $s_\a$ is an orthogonal reflection (with respect to $\langle,\rangle$).  Let
\begin{equation}{\mathrm{COS}}=\{\cos(\pi/m)~|~m \in 
\NM_{\geqslant 2}\} \cup \RM_{\geqslant 1}.
\end{equation}
Assume that $\Pi$ is a subset of $\EC$ with the following properties (i)--(iii):

\medskip

\begin{itemize}\itemth{i} $ \Pi$ is positively independent.
\itemth{ii} For all $\alpha\in \Pi$, $\mpair{\alpha,\alpha}=1$.
\itemth{iii} For all $\alpha,\beta\in \Pi$ with $\alpha\neq \beta$, one has $\mpair{\alpha,\beta}\in -\mathrm{COS}$.\end{itemize}

\medskip

 Let $S:=\mset{s_\alpha\mid \alpha\in \Pi}$, let  $W$ be 
the subgroup of the orthogonal group $O(\EC, \langle, \rangle)$  generated by $S$, 
$$
\Phi:=\{w(\a)~|~w \in W\text{~and~}\a \in \Pi\}, \qquad\Phi_+=\Phi \cap \Bigl(\sum_{\alpha\in \Pi} \RM_{\geqslant 0}~\alpha\Bigr).$$
Then $(W,S)$ is a Coxeter system,  in which the order $m_{s_\alpha,s_\beta}$ of  the product $s_\alpha s_\beta$ for $\alpha,\beta\in \Pi$ is given by 

\begin{equation}\label{order inner prod}  
m_{s_\alpha,s_\beta}=\begin{cases}m, &\text{\rm if $\mpair{\alpha,\beta}=-\cos\DS{\frac{\pi}{m}},\quad m\in 
\Nat_{\geqslant 1}$}\\
\infty, &\text{\rm if $\mpair{\alpha,\beta}\le -1$.} \end{cases} 
\end{equation}
One has 
\begin{equation}
\Phi=\Phi_+ \hskip1mm\dot{\cup}\hskip1mm -\Phi_+.
\end{equation}

\medskip

When the above  conditions hold, we say that $(\Phi,\Pi)$ is a based root system in $(\EC, \langle, \rangle)$ with associated Coxeter system $(W,S)$.  Every Coxeter system is isomorphic to the Coxeter system of some  based root system  (and even to  one with $\mpair{\alpha,\beta}=-\cos\frac{\pi}{m_{s_\alpha,s_\beta}}$ for all $\alpha,\beta\in \Pi$, and with $\Pi$ a basis of $\EC$;  a based root system of this type is called a standard based root system).  All the usual results for standard based root systems which we use in this paper, and their proofs, extend mutatis mutandis to the  based root systems as defined above, unless additional hypotheses are indicated in our statements here (as in Lemma \ref{brink} below, for example).

\medskip

Let us collect some additional basic facts about such based root systems.
\begin{quotation}
{\small 

\begin{lemma} \label{pos root} For  $w\in W$ and $\alpha\in \Phi_+$, one has 
$w(\alpha)\in \Phi_+$ iff $\ell(ws_\alpha)>\ell(w)$.
\end{lemma}

\begin{lemma}\label{dyer can}
Let $\D \subseteq \Phi_+$, let $T'=\{s_\a~|~\a \in \D\}$ and let $W'$ 
denote the subgroup of $W$ generated by $T'$. Then $T'$ is the set of 
canonical Coxeter generators of $W'$ if and only if 
$-\langle \a, \b \rangle \in {\mathrm{COS}}$ for all $\a$, $\b \in \D$ such 
that $\a \neq \b$.
\end{lemma}

\begin{proof}
See \cite[(4.4)]{dyer}
\end{proof}

\begin{lemma}[Brink]\label{brink} 
Let $\g \in \Phi_+$. Then one may write $\g=\sum_{\a \in \Pi} c_\a \a$
with  $c_\a/2 \in {\mathrm{COS}}$ for all $\a \in \Pi$. In particular, 
if $c_\a \not\in \{0,1\}$, then $c_\a \ge \sqrt{2}$. If $\Pi$ is linearly independent, the $c_\a$ are uniquely determined by the conditions  $\g=\sum_{\a \in \Pi} c_\a \a$  and $c_\a\in \real$.
\end{lemma}

\begin{proof}
For the standard reflection representation,  for which $\Pi$ is linearly independent, see \cite[Proposition 2.1]{brink}. A quick sketch
in general is as follows. One checks the statement for dihedral
Coxeter systems (for which $\Pi$ is automatically linearly independent) by  direct calculations (see \cite[(4.1)]{dyer}). Then in general, a standard proof
(loc cit)  of Lemma \ref{pos root} by reduction to rank two shows that there is some choice of root coefficients 
$c_\a$ such that all $c_\a$ are expressible as polynomials with non-negative integer coefficients in the (positive) root coefficients for  rank two  standard parabolic subgroups, and the result follows.
\end{proof}

\begin{lemma}\label{root prod and length} Let $\beta\in\Pi$ and $\alpha\in\Phi_{+}\setminus\set{\beta}$.
Then \begin{enumerate}\itemth{a} $s_{\beta}(\alpha)\in \Phi_{+}$ and $s_{s_{\beta}(\alpha)}=s_{\beta}s_{\alpha}s_{\beta}$.

\itemth{b} $\ell(s_{\beta}s_{\alpha}s_{\beta})$ is equal to $\ell(s_{\alpha})+2$, $\ell(s_{\alpha})$
or $\ell(s_{\alpha})-2$ according as whether $\mpair{\alpha,\beta}<0$, $\mpair{\alpha,\beta}=0$
or $\mpair{\alpha,\beta}>0$. If $\mpair{\alpha,\beta}=0$, then $s_{\beta}s_{\alpha}s_{\beta}=s_{\alpha}$.
\end{enumerate}
\end{lemma}

\begin{proof} Part (a) is well-known, and so is (b)  in the special  case of linearly independent  simple roots. One may also verify   (b) for dihedral Coxeter systems by  direct calculation
(using  \cite[(4.1)]{dyer}) again, for instance). In general, (b) may be reduced to the dihedral case as follows. Let $W':=\mpair{s_{\alpha},s_{\beta}}$, $T'=\chi(W')$
and $l'$ be the length function of $(W',T')$. In case $\mpair{\alpha,\beta}=0$, then by the dihedral case, $s_{\beta}s_{\alpha}=s_{\alpha}s_{\beta}$ and so $\ell(s_{\beta}s_{\alpha}s_{\beta})=\ell(s_{\alpha})$. In case  $\mpair{\alpha,\beta}<0$, then by the dihedral case, one has $\ell'(s_{\beta})<\ell'(s_{\beta}s_{\alpha})<\ell'(s_{\beta}s_{\alpha}s_{\beta})$.
Hence  by Lemma \ref{dyer X} (c), one has $\ell(s_{\beta})<\ell(s_{\beta}s_{\alpha})<\ell(s_{\beta}s_{\alpha}
s_{\beta})$ and thus $\ell(s_{\beta}s_{\alpha}s_{\beta})=\ell(s_{\alpha})+2$ as required.
The remaining case $\mpair{\alpha,\beta}>0$ follows from (a) and the second case
applied to 
$\alpha':=s_{\beta}(\alpha)$ in place of $\alpha$, since $\mpair{\alpha',\beta}<0$.  \end{proof}
 The chief technical advantage of the class of based root systems is explained  by Lemma \ref{subsystem} below. It  follows  from the definition and  previously given facts about  based root systems (especially Lemma \ref{dyer can} and \eqref{order inner prod}).

\medskip
  
\begin{lemma}\label{subsystem} 
Let $(\Phi, \Pi)$ be a based root system in $(\EC, \langle,\rangle)$, 
with associated Coxeter system $(W,S)$. Let $W'$ be a reflection subgroup 
of $(W,S)$ and set $S':=\chi(W')$. 
Let $\Psi:=\mset{\a\in \Phi\mid s_\alpha\in W'}$ and 
$\Delta:=\mset{\a\in \Phi_{+}\mid s_\alpha\in S'}$. Then 
$(\Psi, \Delta)$ is a based root system in $(\EC, \langle,\rangle)$ 
with associated Coxeter system $(W',S')$.
\end{lemma}

\medskip

\noindent {\sc Remark -}  Note that even if  $(\Phi,\Pi)$ is a standard based root system and   $S'$ is finite, the elements of $\Delta$ need not be linearly independent, and for elements $\alpha,\beta$ of $\Delta$ such that $s_\alpha s_\beta$ has infinite order, one may have $\mpair{\alpha,\beta}<-1$. Thus, the lemma fails for the class of standard based root systems in two important respects.  
 
 }
\end{quotation}

\bigskip 
Although not logically required in this paper, we include the following  alternative proof of Theorem \ref{gal2} and  part of  Theorem \ref{dyer} using based root systems, because of its intrinsic  interest and since the general method of proof may be applicable in other situations. Precisely, we shall prove here the following:
\begin{theorem}\label{alternative} Let $(W,S)$ be a Coxeter system. Let $S=I\dot \cup J$ be a partition of $S$ as in $\text{\rm Theorem \ref{dyer}}$, and define $\Wti,\Jti,\Mti$ as in $\text{\rm Theorems \ref{dyer}}$ and $\text{\rm \ref{gal2}}$.
Then $(\Wti, \Jti)$ is a Coxeter system with Coxeter matrix $\Mti$ and 
$\Jti=\chi(\Wti)$ is the canonical set  of Coxeter generators of $\Wti$.
\end{theorem}
\begin{proof} 
We assume without loss of generality that $(W,S)$
is the Coxeter system associated to a  based root system $(\Phi,\Pi)$ such that $\Pi$ is linearly independent. We keep other notation as above. 

\medskip

Let $\Pi_K:=\mset{\alpha\in \Pi\mid s_\alpha\in K}$ for any $K\subseteq S$. By  \eqref{even} and \eqref{order inner prod}, the assumption that no element of $I$ 
is conjugate to any element of $J$ is therefore equivalent to 
the assertion that if $\gamma\in \Pi_I$ and $\delta\in \Pi_J$, 
then $\mpair{\gamma, \delta}$ is either of the form
$\mpair{\gamma,\delta}=-\cos\frac{\pi}{2m}$ for some $m\in\Nat_{\geq 1}$ or satisfies $\mpair{\gamma,\delta}\le -1$. In particular,
\begin{equation}\label{inf}
\text{\it If $\g \in \Pi_I$ and $\d \in \Pi_J$, then 
$\mpair{\g,\d} \le \DS{-\frac{\sqrt{2}}{2}}$.}
\end{equation} 
Now, let
$$\Piwt=\{w(\a)~|~w \in W_I\text{~and~}\a \in \Pi_J\}.$$
Then $\Piwt \subseteq \Phi_+$ by Lemma \ref{pos root}, and $\Jti=\{s_\a~|~\a \in \Piwt\}$. 

\medskip

By Lemma \ref{dyer can} and \eqref{order inner prod}, 
it is sufficient to show that, if 
$\alpt$, $\bett \in \Piwt$ are such that $\alpt\neq \bett$ 
and if $\sti=s_\alpt$ and $\tti=s_\bett$, then 
$$\begin{cases}
\mpair{\alpt,\bett} = -\cos\DS{\Bigl(\frac{\pi}{\mti_{\sti,\tti}}\Bigr)} & 
\text{if $\mti_{\sti,\tti} < \infty$,}\\
\mpair{\alpt,\bett} \le -1 & 
\text{if $\mti_{\sti,\tti} = \infty$,}\\
\end{cases}\leqno{(*)}$$
For this, let $s=\nu(\sti)$, $t=\nu(\tti)$ and let $x$, $y \in W_I$ be such 
that $\sti=xsx^{-1}$ and $\tti=yty^{-1}$. Let $\a=x^{-1}(\alpt)$, 
$\b=y^{-1}(\bett)$ and $w=f(\sti,\tti)$. Then $\a$, $\b \in \Pi_J$, 
$s=s_\a$, $t=s_\b$ and 
$$\mpair{\alpt,\bett} = \mpair{\a,w(\b)}.$$
Indeed, if we write $x^{-1}y=awb$ with $a \in W_{I \cap s^\perp}$ and 
$b \in W_{I \cap t^\perp}$, then 
$$\mpair{\alpt,\bett}=\mpair{x(\a),y(\b)}=\mpair{\a,awb(\b)}=
\mpair{a^{-1}(\a),wb(\b)}=\mpair{\a,w(\b)}.$$

We shall now need the notion of the support of a positive root. If $\d \in \Phi_+$, 
write $\d=\sum_{\g \in \Pi} c_\g \g$ with $c_\g \ge 0$: 
the {\it support} $\supp(\delta)$ of $\delta$ is 
the subset of $\Pi$ defined by 
$\supp(\delta):=\mset{\g\in \Pi\mid c_\g\neq 0}$. 
This is well-defined since we have assumed $\Pi$ is linearly independent.
We recall the following facts:
\begin{quotation}
{\small
\begin{lemma}\label{support} Let $\d \in \Phi_+$ and
$A:=\mset{s_{\g}\mid \g\in \supp(\d)}$. Then
\begin{enumerate} \item 
 $s_\d \in W_{A}$.
\item The full subgraph of the Coxeter graph of $(W,S)$ with vertex set 
$A$ is connected. \end{enumerate}
\end{lemma}
\begin{proof}
We prove (a)--(b) by induction on $l(s_{\delta})$. If $l(s_{\delta})=1$, then $\delta\in\Pi$
and (a)--(b) are clear. Otherwise, write $\d=\sum_{\alpha\in \Pi}c_{\alpha}\alpha$
with all $c_{\alpha}\geq 0$. Since $0<1=\mpair{\alpha,\d}=\sum_{\alpha}c_{\alpha}\mpair{\alpha,\d}$ there is some $\alpha\in \supp(\d)$ with $\mpair{\alpha,\d}>0$. Note $\a\neq\d$ since $\d\not\in \Pi$, so $\g:=s_{\a}(\d)\in \Phi_{+}$.  By Lemma \ref{root prod and length}, $l(s_{\g})=l(s_{\d})-2$.
Let $B:=\mset{s_{\b}\mid \b\in \supp(\d)}$. By induction, $s_{\gamma}\in W_{B}$ and 
the full subgraph of the Coxeter graph of $(W,S)$ on vertex set $B$ is connected.
Since $\delta=s_{\a}(\g)=\g+\mpair{\d,\a}\a$, we have $\supp(\d)=\sup(\g)\cup\set{\a}$
and $A=B\cup\set {s_{\a}}$. Since $0>-\mpair{\a,\d}=\mpair{\a,\g}$, an argument like that above shows that there is some $\b\in \supp(\g)$ with $\mpair{\a,\b}<0$. Therefore $s_{\alpha}$ is joined by an edge in the Coxeter graph of $(W,S)$ to $s_{\beta}\in B$, completing the inductive proof of (b). Since  $s_{\d}=s_{\a}s_{\g}s_{\a}\in  W_{A}$,  the inductive proof of (a) is also finished
\end{proof}
}
\end{quotation}
\medskip

Now, let $\G$ be the unique subset of $\Pi_I$ such that 
$\supp(w(\b))=\G \cup \{\b\}$ and  set $I_\G=\{s_\g~|~\g \in \G\}$. We write
$$w(\b)=\b+\sum_{\g \in \G} c_\g \g,$$
with $c_\g >0$. 
In order to prove $(*)$, we shall need the following lemmas:

\begin{quotation}
{\small \begin{lemma}\label{racine de 2}
Let $\g \in \Pi_I$. Then:
\begin{itemize} 
\itemth{a} If $\g \in \G$, then $c_\gamma\geq \sqrt 2$.

\itemth{b} If $s_\gamma$ appears in a reduced expression for $w$ 
and  $\mpair{\beta,\ck \gamma}\neq 0$, then $\g \in \G$ and either 
$c_\gamma=-\mpair {\beta,\ck\gamma}$ or $c_\gamma\geq 2\sqrt 2$.
\end{itemize}
\end{lemma}

\begin{proof}
We shall argue by induction on $\ell(w)$. If $\ell(w)=0$, 
this is vacuously true. Otherwise, write $w=xs_\delta$ where 
$\delta\in \Pi_I$ and $\ell(x)<\ell(w)$.
We have $s_\delta(\beta)=\beta+c\delta$ where 
$c:=-\mpair{\beta,\ck\delta}$. If $c=0$, then $w(\beta)=x(\beta)$ 
and the desired result follows by induction. Otherwise, 
$c\geq \sqrt 2$ and $w(\beta)=x(\beta)+c x(\delta)$. 
Note $x(\delta) \in \Phi_+$  by Lemma \ref{pos root} since 
$\ell(xs_\delta) > \ell(x)$. Using the inductive hypothesis (a)--(b)
for  $x(\beta)$ and  Lemma \ref{brink} for $x(\delta)$, 
one  gets (a)--(b) for $w(\b)$ (for (b), one has to consider
the cases $\gamma=\delta$, $\gamma\neq \delta$ separately, and note that 
if $s_\delta$ does not appear in a reduced expression for $x$, then 
the coefficient of $\delta$ in $x(\delta)$ is $1$). 
\end{proof}

\begin{lemma}\label{orthogonal}
If $I_\G \subseteq s^\perp$, then $w=1$.
\end{lemma}

\begin{proof}
Indeed, if $I_\G \subseteq s^\perp$, then Lemma \ref{support}(a) implies that
we have  $w t w^{-1} \in W_{\{t\} \cup (I \cap s^\perp)}$. In other words, 
$wt \in W_{\{t\} \cup (I \cap s^\perp)} w$. But $w$ has minimal 
length in $W_{\{t\} \cup (I \cap s^\perp)} w$ by construction, so 
$wt$ does not have minimal length in 
$W_{\{t\} \cup (I \cap s^\perp)} wt$. By Deodhar's Lemma, 
there exists $u \in \{t\} \cup (I \cap s^\perp)$, 
such that $wt=uw$. In other words, $u=wtw^{-1}$ and, since no element 
of $I$ is conjugate to $t$, we have $u=t$ and 
$wt=tw$. So $w \in W_{I \cap t^\perp}$ (see Lemma \ref{centralisateur}), 
and thus $w=1$ because $w$ has minimal length in $w W_{I \cap t^\perp}$. 
\end{proof}
}
\end{quotation}
 
\medskip

We shall now prove $(*)$ by a case-by-case analysis:

\medskip

$\bullet$ If $s=t$ and $w \in I$, let us write $w=s_\g$ with 
$\g \in \Pi_I$. Then $\a=\b$, $\mti_{\sti,\tti}=m_{s,w}/2$ 
and $w(\b)=\a-\mpair{\a,\g^\vee}\g$, so 
$$\mpair{\a,w(\b)}=\mpair{\a,\a}-2 \mpair{\a,\g}^2 = 
1 - 2 \cos^2\Bigl(\frac{\pi}{m_{s,w}}\Bigr)=
-\cos\Bigl(\frac{2\pi}{m_{s,w}}\Bigr),$$
as required. 

\medskip

$\bullet$ If $s=t$ and $\ell(w) \ge 2$, then $\mti_{\sti,\tti}=\infty$. 
First, note that
$$I_\G \nsubseteq s^\perp$$
(see Lemma \ref{orthogonal}). Moreover,
$$\mpair{\a,w(\b)}=\mpair{\a,\b} + \sum_{\g \in \G} c_\g \mpair{\a,\g}
= 1 +\sum_{\stackrel{\SS{\g \in \G}}{s_\g \not\in s^\perp}} 
c_\g \mpair{\a,\g}$$
But, if $\g \in \G$ is such that $s_\g \not\in s^\perp$, then 
$c_\g \ge \sqrt{2}$ by Lemma \ref{racine de 2} (a) and 
$\mpair{\a,\g}=-\cos(\pi/m_{s,s_\g}) \le -\sqrt{2}/2$ by \eqref{inf} 
(since $\a \in \Pi_J$ and $\g \in \Pi_I$). 
Therefore, 
$$\mpair{\a,w(\b)}\le 1 - |I_\G \setminus s^\perp|.$$
So, if $|I_\G \setminus s^\perp| \ge 2$, then $\mpair{\a,w(\b)}\le -1$, 
as required.

So we may assume that $I_\G \setminus s^\perp = \{s_\g\}$ with 
$\g \in \G$. Note that $\mpair{\a,w(\b)}=1-c_{\g}\mpair{\a,\g}$ and that  $s_\g$ appears in a reduced expression of $w$. 
By Lemma \ref{racine de 2} (b), two cases may occur:
\begin{itemize}
\item[-] If $c_\g \ge 2\sqrt{2}$ then, since 
$\mpair{\a,\g} \le -\sqrt{2}/2$ (again by the inequality \eqref{inf}), 
we get that $\mpair{\a,w(\b)}\le -1$, as required.

\item[-] If $c_\g = -\mpair{\b,\g^\vee}$ then 
$$ \supp(s_{\g}w\b)=\supp(w\b)\setminus\set{\g}= (\Gamma\setminus\set{\g})\cup\set{\b}.$$ But no element of $\mset{s_{\d}\mid \d\in \Gamma\setminus\set{\g}}$ is connected
to $s_{\beta}$ in the Coxeter graph of $(W,S)$, so by Lemma \ref{support} (b) we get that
$\Gamma=\set{\gamma}$, $\supp(s_{\g}w\b)=\set{\b}$ and so $s_{\g}w\b=\b$. 
Hence $s_{\g} w \in W_{I \cap t^\perp}$. By 
Deodhar's Lemma, this can only happen 
if $w=s_\g$, which contradicts the fact that $\ell(w) \ge 2$.
\end{itemize}

\medskip

$\bullet$ If $s\neq t$ and $w=1$, then $\mti_{\sti,\tti}=m_{s,t}$ and 
$$\mpair{\alpt,\bett}=\mpair{\a,\b}=-\cos\Bigl(\frac{\pi}{m_{s,t}}\Bigr),$$
as required. 

\medskip

$\bullet$ If $s \neq t$ and $w \neq 1$, then $\mti_{\sti,\tti}=\infty$. 
First, note that
$$I_\G \nsubseteq s^\perp$$
(see Lemma \ref{orthogonal}). So let $\g \in \G$ be such that 
$\mpair{\a,\g} \neq 0$. Then $c_\g \ge \sqrt{2}$ by Lemma \ref{racine de 2} 
and, by \eqref{inf}, we have  
$\mpair{\a,\g} \le -\sqrt{2}/2$ (since 
$\a \in \Pi_J$ and $\g \in \Pi_I$). So 
$$\mpair{\a,w(\b)} \le  \mpair{\a,\b} -1+ 
\sum_{\g' \neq \g} c_{\g'} \mpair{\a,\g'} \le -1$$
because $\mpair{\a,\b} \le 0$ and $\mpair{\a,\g'} \le 0$ 
for all $\g' \in \Pi_I$.

The proof of Theorem \ref{alternative} is now complete.
\end{proof}

\bigskip

The final main  result of this section is  a geometric variant  (Theorem \ref{ext root} below) of Theorem \ref{external}. To formulate it,  we shall require  the notions of 
automorphisms,   fundamental chamber and Tits cone of a based root system. 
The latter two are principally of interest when the form $\langle, \rangle$ 
on $\EC$ is non-degenerate, but our application won't require this 
(and non-degeneracy can always be achieved by enlarging the space $\mc{E}$ 
and extending the form $\langle, \rangle$, anyway). 

\medskip

Let $(\Phi, \Pi)$ be a based root system in $(\EC, \langle,\rangle)$, with associated Coxeter system $(W,S)$. By an automorphism of $(\Phi,\Pi)$, we mean an element $\theta$  of $O(\EC, \langle,\rangle)$ which restricts to  permutations of both  $\Pi$ and  $\Phi$.  For example, in the setting of the proof of Theorem \ref{alternative}, $W_I$ acts naturally as a group of based root  system automorphisms of the 
based root system attached  by Lemma  \ref{subsystem} to $\widetilde W$.

\medskip

  In general, we define the fundamental chamber of $(W,S)$ on $\mc{E}$ to be the subset  $\CC=\CC_{(W,S)}:=\mset{\rho \in \EC\mid \mpair{\alpha,\rho}\geq 0 \text{ for all }\alpha\in \Pi}$ of $\EC$, and we call  $\XC=\XC_{(W,S)}=W\CC:=\cup_{w\in W}w(\CC)$ the Tits cone.
The most basic properties of $\CC$ and $\XC$ (see \cite{bourbaki}) are recalled in the following Lemma.

\medskip
\begin{quotation}
{\small 
\begin{lemma}\label{Tits} \begin{itemize} \itemth{a} $\XC=\mset{\rho\in \EC\mid \vert \mset{\alpha\in \Phi_+\mid \mpair{\alpha,\rho}<0}\vert<\infty}$. In particular, $\XC$ is a convex cone in $\EC$.
\itemth{b} Any $W$-orbit on $X$ contains a unique element of $\CC$.
\itemth{c} For $\a\in \CC$, the stabilizer $W_\a:=\mset{w\in W\mid w(\a)=\a}$ of $\a$ is the standard parabolic subgroup of $W$ generated by
$\mset{s\in S\mid s(\a)=\a}$.
\end{itemize}\end{lemma}
}
\end{quotation}

Now we may state:

\medskip

\begin{theorem}\label{ext root} Let $(\Psi,\Delta)$ and $(\widetilde \Phi,\widetilde \Pi)$  be two based root systems
in $(\EC, \langle, \rangle)$ with associated Coxeter systems $(W',I)$ and $(\Wti, \Jti)$  respectively.
Let $\CC:=\CC_{(W',I)}$ and $\XC:=\XC_{(W',I)}$.
Assume that $W'(\widetilde \Pi)\subseteq \widetilde\Pi$. Then $W'$ acts as a group of based root system automorphisms of $(\widetilde \Phi,\widetilde \Pi)$ and  also as a group of automorphisms of the Coxeter system $(\Wti,\Jti)$.
Let $W$ denote the subgroup of $O(\EC, \langle, \rangle)$ generated by the   subset  $W'\cup \Wti$.
Then  $W=\Wti\rtimes W'$.
 Under these assumptions, the following conditions are equivalent:
\begin{itemize}\itemth{i} There is a  based root system $(\Phi,\Pi)$ with
$\Delta\subseteq \Pi\subseteq \Delta\cup\widetilde \Pi$ and  $\widetilde{\Pi}=W'(\Pi\setminus \Delta)$.
\itemth{ii}
 $\Delta \cup \widetilde{\Pi}$ is positively independent and
$\widetilde{\Pi}\subseteq- \XC$.
\end{itemize} 
 Assume  conditions $\text{\rm (i)--(ii)}$ hold. Then  $\Pi=\Delta ~\dot\cup~ (\widetilde \Pi \cap -\CC)$
 (so $(\Phi,\Pi)$ is uniquely determined in $\text{\rm (i)}$),  
   $\Psi\cup\widetilde  \Phi\subseteq \Phi$, 
 and $\widetilde{\Phi}_+\subseteq -\XC$.
  Set $S:=\mset{s_\alpha\mid \alpha\in \Pi}$ and $J=S\setminus I$. Then $(W,S)$ is the Coxeter system associated to the based root system $(\Phi,\Pi)$, 
  $\Jti=\mset{wsw^{-1}\mid w\in W', s\in J}$, and no element of $I$ is conjugate to any element of $J$. The semidirect  product decomposition 
  $W=\Wti\rtimes W'$ is that attached by $\text{\rm Theorem \ref{dyer}}$  to the 
  subsets $I$ and $J$ of $S$.
\end{theorem}

\begin{proof}  
For any $\theta\in O(\EC, \langle, \rangle)$ and $\alpha\in \EC$ with
$\mpair{\alpha,\alpha}=1$, one has $\mpair{\theta(\alpha),\theta(\alpha)}=1$ and $s_{\theta(\alpha)}=\theta s_\alpha\theta ^{-1}$. Assume further that $\theta (\Pi')\subseteq \Pi'$. Then this implies that $\Jti$, and hence $\Wti$, is stable under conjugation by $\theta $, and so $\theta $ acts as an automorphism of $(\Wti,\Jti)$.
If $\alpha\in \widetilde \Phi$, we can write $\alpha=x(\beta)$ for some $\beta\in \widetilde \Pi$
and $x\in \widetilde W$. Then $\theta (\alpha)=\theta x(\beta)=(\theta x\theta ^{-1})(\theta (\beta))\in \widetilde \Phi$
since $\theta x\theta ^{-1}\in \widetilde W$ and $\theta (\beta)\in \widetilde \Pi$. 
Hence $\theta(\widetilde \Phi)\subseteq \widetilde  \Phi$. For $\gamma\in\widetilde  \Phi_+$, we may write $\gamma=\sum_{\alpha\in \widetilde \Pi}c_\alpha \alpha$
with all $c_\alpha\geq 0$. Then $\theta(\gamma)=\sum_{\alpha\in \widetilde \Pi} c_\alpha \theta(\alpha)\in \widetilde \Phi_+$ since all $\theta(\alpha)\in \widetilde \Pi$, showing that $\theta(\widetilde \Phi_+)\subseteq \widetilde  \Phi_+$.

\medskip

The above all applies with $\theta\in W'$, proving that $W'$ acts as automorphisms of $(\Wti,\Jti)$ and $(\widetilde\Phi,\widetilde \Pi)$. In particular, $W'$ normalizes $\Wti$. 
If $w\in W'$, then $w$ permutes $\widetilde \Phi_+$. If $w\in W'\cap \Wti$, this implies that $\tilde \ell(w)=0$
(since $w$ makes no element of $\widetilde \Phi_+$ negative) so $w=1_{W'}$. From the above, we see that
$W=W'\Wti=\Wti\rtimes W'$ as claimed. We also see that $\Psi\cap \widetilde \Phi=\emptyset$,
for if $\alpha\in\Psi\cap \widetilde \Phi$, then $s_\alpha\in W'\cap \Wti=\set{1_{W'}}$
which is  a contradiction. From this, one sees further that $\widetilde \Phi$ is stable under the $W$-action on $\EC$ and hence that  no element of $\Psi$ is $W$-conjugate to any element of $\widetilde \Phi$. 

\medskip

Now suppose that  the assumptions of (i) hold. 
 Since $\Pi_+$ is positively independent, it follows that $\Phi_+$ is positively independent, and hence so also   is the subset $\Delta\cup \widetilde \Pi$ of $\Phi_+$.
 Let $\alpha\in \Pi\setminus \Delta\subseteq \widetilde \Pi$.
 Since $\alpha\not\in \Delta$, we have $\mpair{\alpha,\beta}\in -\textrm{COS}$ for all $\beta\in \Delta$. In particular, $\mpair{\alpha,\beta}\le 0$ so $\alpha\in -\CC$.
 Thus, $\Pi\setminus \Delta\subseteq -\CC$.
 Hence  \[\widetilde \Pi=W'(\Pi\setminus \Delta)\subseteq W' (-\CC)=-\XC.\]
 Therefore $\widetilde \Phi_+\subseteq -\XC$ also since $\XC$ is a convex cone.
 Since every $W'$-orbit on $-\XC$ contains a unique point of $-\CC$, $\widetilde \Pi$ is $W'$-stable
 and $\widetilde \Pi\subseteq W'(\Pi\setminus \Delta)$,
 it follows using Lemma \ref{Tits} (b) that $\Pi\setminus \Delta=\widetilde \Pi\cap -\CC$.
 Observe also that  we  have $\Psi\cup \widetilde \Phi\subseteq \Phi$ and so 
  \[W=\mpair{s_\alpha\mid\alpha\in  \Psi\cup\widetilde \Phi}\subseteq
 \mpair{s_\alpha\mid \alpha\in \Phi}=\mpair{s_\alpha\mid \alpha\in \Pi}
 \subseteq \mpair{s_\alpha\mid \alpha\in \Delta\cup\widetilde \Pi}=W\] 
 which implies that if (i) holds, then  the Coxeter system associated to $(\Phi,\Pi)$
is  $(W,S)$ 
 where $S:=\mset{s_\alpha\mid \alpha\in \Pi}$.
 
 \medskip
 
 Now suppose that the assumptions of (ii) hold. Set $\Pi=\Delta ~\dot\cup~ (\widetilde \Pi \cap -\CC)$. 
Clearly,  $\Delta\subseteq \Pi\subseteq \Delta\cup\widetilde \Pi$.  We also have $\widetilde{\Pi}=W'(\Pi\setminus \Delta)$
since $\widetilde{\Pi}\subseteq- \XC$ and $\widetilde{\Pi}$ is $W'$-stable.
Let $S:=\mset{s_\alpha\mid \alpha\in \Pi}$ and $W''$ be the subgroup generated by $S$. It is clear
$W''$ contains $W'$ and $s_\alpha$ for $\alpha\in \Pi\setminus \Delta$, so it also contains
$ws_\alpha w^{-1}$ for such $\alpha$ and all $w\in W'$. That is,  $W''$ contains the group generated by
$s_\beta$ for all $\beta\in W'(\Pi\setminus \Delta)=\widetilde \Pi$. So $W''\supseteq W'\Wti=W$. 
But clearly, $S\subseteq W$, so $W''=W$. Let $\Phi=W\Pi$.

\medskip

Since 
 $\Delta \cup \widetilde{\Pi}$ is positively independent, to show that $(\Phi,\Pi)$ is a based root system, it will suffice to show that if $\alpha,\beta\in \Pi$ with $\alpha\neq \beta$, then $c:=-\mpair{\alpha,\beta}\in \mathrm{COS}$.  If both $\alpha,\beta$ are in $\Delta$, or both are in $\widetilde \Pi$, this follows since $(\Psi,\Delta)$ and $(\widetilde \Phi,\widetilde \Pi)$  are based root systems. The remaining case is that, say,  $\alpha\in \Delta$ and $\beta\in \widetilde \Pi$. We show that in this case, $c\in\mathrm{COS}':= \set{-\cos \pi/2m\mid m\in \Nat_{\geq 1}}\cup \real_{\geq 1}$.
 We have $c\geq 0$ since $\beta\in -\CC$. Also, $s_\alpha(\beta)=\beta+2c\alpha\in \widetilde \Pi$.
 If $s_\alpha(\beta)=\beta$, then $c=0\in \mathrm{COS}'$. Otherwise, $s_\alpha(\beta)\neq \beta$ are both in $\widetilde \Pi$, so $d:=-\mpair{s_\alpha(\beta),\beta}\in \mathrm{COS}$ because 
 $(\widetilde \Phi,\widetilde \Pi)$  is a  based root system.  But $d=-\mpair{\beta+2c\alpha,\beta}=
 -1+2c^2$. So $c=\sqrt{\frac{d+1}{2}}$ with $d\in \mathrm{COS}$. 
 If $d\geq 1$, say $d=\cosh \lambda$ where $\lambda\in \real$,  then $c=\cosh\frac \lambda 2\geq 1$ so $c\in \mathrm{COS}'$. Otherwise, $d=\cos\frac{\pi}{m}$ for some $m\in\Nat_{\geq 2}$, so $c=\cos\frac{\pi}{2m}\in \mathrm{COS}'$. This shows that (ii) implies (i).
 Note that $J=S\setminus I=\mset{s_\alpha\mid \alpha\in \Pi\setminus \Delta}$. The argument above also shows that no element of $I$ is $W$-conjugate to any element of $J$.

\medskip
 
 Assuming that (i) and (ii) both hold,
 the remaining assertions of the Theorem follow directly  from the  consequences of (i)--(ii) proved above. 
 \end{proof}

\bigskip

\section{Affine reflection groups}

\medskip


Let $E$ be a finite dimensional affine space over $\RM$ and assume that 
the underlying vector space $\EC$ is endowed with a positive definite 
scalar product $\mpair{,}$. If $H$ is an hyperplane in $E$, we denote by 
$s_H$ the orthogonal reflection with respect to $H$. 

Let $\AG$ be an (affine) hyperplane arrangement in $E$ and let $W$ 
be the subgroup of $O(E,\mpair{,})$ generated by $(s_H)_{H \in \AG}$. 
As in \cite[Chapter V, \S 3]{bourbaki}, we assume that the 
following hypothesis are satisfied:
\begin{quotation}
\begin{itemize}
\itemth{D1} $W$ stabilizes $\AG$.

\itemth{D2} The group $W$, endowed with the discrete topology, acts 
properly on $E$.
\end{itemize}
\end{quotation}
We can then define 
the notions of $\AG$-chambers, $\AG$-walls, $\AG$-facets, $\AG$-faces 
as defined in \cite[Chapter V, \S 1]{bourbaki}.
We fix an $\AG$-chamber $C$ and we denote by $\D$ the set of 
$\AG$-walls of $C$. 
Let $S=\{s_H~|~H \in \D\}$. Then $(W,S)$ is a Coxeter system 
and $\overline{C}$ (the closure of $C$) is a fundamental domain 
for the action of $W$ on $E$ 
(see \cite[Chapter V, \S3, Theorems 1 and 2]{bourbaki}). 

We still assume that we have a partition $S = I \dotcup J$ 
such that no element in $I$ is $W$-conjugate to an element in $J$ 
and we keep the notation of the previous sections. We set 
$$T=\{s_H~|~H \in \AG\},\quad\Delt=\{H \in \AG~|~s_H \in \Jti\},$$
$$\Tti=T \cap \Wti\quad\text{and}\quad\AGt=\{H \in \AG~|~s_H \in \Tti\}.$$
Then $\AGt$ is an hyperplane arrangement satisfying (D1) and (D2). Let
$\Cti$ be the unique $\AGt$-chamber containing $C$. Then $\Delt$ is the 
set of $\AGt$-walls of $\Cti$. We have:

\begin{proposition}\label{c}
$\DS{\overline{\Cti}=\bigcup_{w \in W_I} w(\overline{C})}$.
\end{proposition}

\begin{proof}
Let $\DS{\hat{C}=\bigcup_{w \in W_I} w(\overline{C})}$. 
First, note that $W_I$ stabilizes $\AGt$, so $W_I$ stabilizes 
$\Cti$ (and $\overline{\Cti}$). 
Therefore, $\hat{C} \subseteq \overline{\Cti}$. 

Conversely, let $p \in \overline{\Cti}$. Then there exists 
$w \in W$ such that $w(p) \in \overline{C}$. Write $w=\wti x$ with 
$\wti \in \Wti$ and $x \in W_I$. Then $x(p) \in \overline{\Cti}$ 
and $\wti(x(p)) \in \overline{\Cti}$. Since $\overline{\Cti}$ is 
a fundamental domain for $\Wti$, we get that $\wti(x(p))=x(p)$. 
So $p=x^{-1}(w(p)) \in x^{-1}(\overline{C}) \subseteq \hat{C}$.
\end{proof}

\begin{corollary}\label{parafine}
Let $L$ be a subset of $\Jti$ such that $\Wti_L$ is finite. 
Then there exists a subset $K$ of $S$ and an element $d$ of 
$X_{I \cap K}^I$ such that $W_K$ is finite and 
$\Wti_L=\Wti \cap \lexp{d}{W_K}$.
\end{corollary}

\begin{proof}
Since $\Wti_L$ is finite, there exists $p \in \overline{\Cti}$ 
such that $\Wti_L = \Stab_{\Wti}(p)$ (see \cite[Chapter V, \S 3]{bourbaki}). 
By Proposition \ref{c}, there exists $x \in W_I$ such that 
$x(p) \in \overline{C}$. Let $K$ be the subset of $S$ such that 
$\Stab_W(x(p))=W_K$ (see \cite[Chapter V, \S 3, Proposition 1]{bourbaki}). 
Then $W_K$ is finite and $\Wti_L = \Wti \cap \lexp{x^{-1}}{W_K}$. 
Now, let $d$ be the unique element of minimal length in $x^{-1} W_{I \cap K}$. 
Then $d \in X_{I \cap K}^I$ and $\Wti_L = \Wti \cap \lexp{d}{W_K}$.
\end{proof}

\begin{corollary}\label{tres affine}
Assume that $(W,S)$ is an irreducible affine Weyl group and that 
$J \neq \vide$. Then all the irreducible components of $\Wti$ are affine. 
\end{corollary}

\begin{proof}
Since $(W,S)$ is affine and irreducible and $I \varsubsetneq S$, 
the group $W_I$ is finite. Therefore, $\overline{\Cti}$ is compact 
and the result follows.
\end{proof}

\bigskip

\noindent{\sc Remark - } The two previous corollaries could have been 
shown using the classification and the Table given at the end of this paper. 

\bigskip

\section{Finite Coxeter groups} 

\medskip

In this section, and only in this section, we assume that 
$W$ is {\it finite}. We shall relate here 
the semidirect product decomposition with other constructions 
which are particular to the finite case: invariants, Solomon 
algebra. We first start by an easy result:

\begin{proposition}\label{j=s}
If $(W,S)$ is finite and irreducible and if $J \neq \vide$, then 
$|\Jti|=|S|$.
\end{proposition}

\noindent{\sc Remark - } 
Of course, the above proposition is easily checked using the 
classification (see the Table at the end of this paper). We shall 
provide here a general proof. As it is also shown by this table, 
the proposition is no longer true in general if we do not assume that 
$W$ is finite.

\begin{proof}
We assume without loss of generality that $(W,S)$
is the Coxeter system associated to a  based root system $(\Phi,\Pi)$ such that $\Pi$ is linearly independent. We keep the notation of the 
proof of Theorem \ref{alternative} ($\Pi_J$, $\Piwt$...). 

First, $\Piwt \subseteq \Phi^+$. Let $(\l_\a)_{\a \in \Piwt}$ 
be a family of real numbers such that $\sum_{\a \in \Piwt} \l_\a \a = 0$. 
Let $x=\sum_{\a \in \Piwt} |\l_\a|~\a$. 
Since $\mpair{\a,\b} \le 0$ if $\a$, $\b \in \Piwt$ 
(see $(*)$ in the proof of Theorem \ref{alternative}) 
and since $\mpair{,}$ is positive definite, we get that $x=0$ because 
$$\mpair{x,x} \le 
\langle \sum_{\a \in \Piwt} \l_\a \a,\sum_{\a \in \Piwt} \l_\a \a \rangle =0.$$
But $\Piwt$ is positively independent, so we get that $\l_\a=0$ for 
all $\a \in \Piwt$. Therefore, $\Piwt$ is linearly independent, 
so $|\Piwt| \le |S|$. 

Since $|\Jti|=|\Piwt|$, it remain to show that $|\Piwt| \ge |S|$ or, 
in other words, that $\Piwt$ generates $\EC$. Let $\EC'$ be the subspace 
generated by $\Piwt$. It is $W_I$-stable by definition of $\Piwt$ 
and it is $\Wti$-stable since $\Wti$ is generated by the (orthogonal) 
reflections $(s_\a)_{\a \in \Piwt}$. So $\EC'$ is $W$-stable by Theorem 
\ref{dyer} (a). Since $\EC$ is an irreducible $W$-module and since 
$\Piwt \not= \vide$, we get that $\EC'=\EC$, as expected.
\end{proof}

\noindent{\bf Invariant theory.} 
Keep the notation of the proof of the Proposition \ref{j=s}. 
We view $\EC$ as an algebraic variety over $\RM$. The group 
$W/\Wti\simeq W_I$ acts linearly on the tangent space $\TC$ to $\EC/\Wti$ 
at $0$. Since $\Wti$ is finite and generated by reflections, 
this tangent space has dimension $\dim \EC$ (since $\EC/\Wti$ is smooth). 
Moreover, by \cite[Theorems 3.2 and 3.12, Proposition 3.5]{BBR}, we have:

\begin{proposition}\label{quotient}
The group $W_I$ acts (faithfully) as a reflection group on $\TC$: 
a reflection in $W_I$ acts as a reflection on $\TC$.
\end{proposition}

\noindent{\sc Remark - } In \cite{BBR}, the authors have investigated 
the links between different objects associated to the invariant theory 
of $W$ and $W/\Wti \simeq W_I$: degrees, hyperplane arrangements, 
fake degrees, regular elements...

\bigskip

\noindent{\bf Solomon descent algebra.} 
If $K \subseteq S$ and $L \subseteq \Jti$, we set
$$x_K = \sum_{w \in X_K} w \in \QM W,\qquad
\xti_L=\sum_{w \in \Xti_L} w \in \QM \Wti,$$
$$x_L=\sum_{w \in X_L} w \in \QM W.
\leqno{\text{and}}$$
The {\it Solomon descent algebra} $\Sigma(W)$ of $W$ is defined by 
$$\Sigma(W)=\mathop{\oplus}_{K \subseteq S} \QM x_K$$
(see \cite{solomon}). It turns out that it is a subalgebra of the group 
algebra $\QM W$. Similarly, we set
$$\Sigma(\Wti)=\mathop{\oplus}_{L \subseteq \Jti} \QM \xti_L.$$
We then define a $\QM$-linear map 
$$\restilde : \Sigma(W) \longto \Sigma(\Wti)$$
by 
$$\restilde(x_K)=\sum_{d \in X_{I \cap K}^I} \xti_{\Jti \cap \lexp{d}{W_K}}
\qquad(=\sum_{d \in X_{I \cap K}^I} \xti_{\lexp{d}{K^+}})$$
for all $K \subseteq S$. 

\begin{proposition}\label{morphisme anneaux}
The map $\restilde : \Sigma(W) \to \Sigma(\Wti)$ is a morphism of 
$\QM$-algebras. 
\end{proposition}

\begin{proof}
Let $z=\sum_{w \in W_I} w$. We shall first show that, for $x \in \Sigma(W)$, 
\begin{equation}\label{conjugaison bizarre}
 z\restilde(x)=xz.
\end{equation}
For this, we may assume that $x=x_K$ for some $K \subseteq S$. Let 
$z'=\sum_{w \in W_{I \cap K}} w$. Since 
$$z=\sum_{d \in X_{I \cap K}^I} z' d^{-1},$$
we get 
$$x_K z = (x_K z') \cdot \sum_{d \in X_{I \cap K}^I} d^{-1}.$$
But $W_{I \cap K}$ is the set of elements $w \in W_K$ of minimal length  in 
$w \Wti_{K^+}$. So it follows from Lemma \ref{dyer X} (c) that 
the map 
$$\fonctio{X_K \times W_{I \cap K}}{X_{K^+}}{(x,w)}{xw}$$
is bijective. So $x_K z' = x_{K^+} = z \xti_{K^+}$. Therefore,
$$x_K z=z \sum_{d \in X_{I \cap K}^I} \xti_{\Jti \cap W_K} d^{-1}.$$
But $zd=z$ for all $d \in X_{I \cap K}^I$, so 
$$x_K z=z\sum_{d \in X_{I \cap K}^I} d \xti_{\Jti \cap W_K} d^{-1}.$$
Since $W_I$ acts on the pair $(\Wti,\Jti)$, we have 
$d \Xti_{\Jti \cap W_K} d^{-1} = \Xti_{\Jti \cap \lexp{d}{W_K}}$, 
so \ref{conjugaison bizarre} follows. 

\medskip

Since the map $\ZM \Wti \to \ZM W$, $u \mapsto zu$ is injective, 
we get immediately from \ref{conjugaison bizarre} that $\restilde$ 
is a morphism of rings.
\end{proof}

Note that the group $W_I$ acts by conjugation on the descent 
algebra of $\Wti$.

\begin{corollary}\label{image fixe}
The image of $\restilde$ is $\Sigma(\Wti)^{W_I}$.
\end{corollary}

\begin{proof}
Let $x \in \Sigma(W)$ and $w \in W_I$. Then, by \ref{conjugaison bizarre}, 
$$z\cdot \lexp{w}\restilde(x)=zw\restilde(x)w^{-1} = 
z\restilde(x)w^{-1} = xzw^{-1} = xz = z \restilde{x},$$
so $\restilde(x)=\lexp{w}{\restilde(x)}$. This shows that 
the image of $\restilde$ is contained in $\Sigma(\Wti)^{W_I}$. 

Conversely, we need to show that, for all $L \subseteq \Jti$, 
the element $u=\sum_{w \in W_I} \xti_{wLw^{-1}}$ is in the image 
of $\restilde$. But, by Corollary \ref{parafine}, there exists 
$K \subseteq S$ and $d \in X_{I \cap K}^I$ such that 
$\Wti_L = \Wti \cap \lexp{d}{W_K}$. Then $L=d K^+ d^{-1}$ and 
$$\restilde(x_K) = \sum_{x \in X_{I \cap K}^I} \xti_{\lexp{x}{K^+}}
= \frac{1}{|W_{I \cap K}|} u,$$
as desired. 
\end{proof}

The Solomon descent algebra $\Sigma(W)$ 
is endowed with a morphism of $\QM$-algebras 
$\theta : \Sigma(W) \longto \QM \Irr W$, where $\QM \Irr W$ denotes 
the algebra of $\QM$-linear combinations of irreducible characters 
of $W$ (with usual product). The map $\th$ is defined by 
$$\th(x_K)=\Ind_{W_K}^W 1_K$$
for all $K \subseteq S$ (here, $1_K$ denotes the trivial character of $W_K$). 
Similarly, we have a morphism of $\QM$-algebras 
$\tilde{\th} : \Sigma(\Wti)\to \QM \Irr \Wti$ and the Mackey formula 
shows immediately that the diagram
\begin{equation}\label{diag}
\diagram
\Sigma(W) \rrto^{\DS{\th}} \ddto_{\DS{\restilde}} && 
\QM\Irr W \ddto^{\DS{\Res_\Wti^W}} \\
&&\\
\Sigma(\Wti) \rrto^{\DS{\tilde{\th}}} && \QM\Irr \Wti
\enddiagram 
\end{equation}
is commutative.

\bigskip

\noindent{\sc Remark - } In \cite[\S 5.2]{BP}, the authors have 
defined the map $\restilde$ whenever $W$ is of type $B_n$ and 
$\Wti$ is of type $D_n$ (it was denoted by ${\mathrm{Res}}_n$). 
In this particular case, Proposition \ref{morphisme anneaux}, 
Corollary \ref{image fixe} and the commutativity of the diagram 
\ref{diag} have been shown in \cite[Proposition 5.9]{BP}. 

\bigskip

\section{Examples}
We shall  describe in detail some examples of (internal)   semidirect 
product decompositions of 
Coxeter systems $(W,S)$. If $\D$ is a Coxeter graph, we shall denote by 
$W(\D)$ the 
associated Coxeter group. In the following table, we have drawn the diagram of 
$(W,S,{\boldsymbol{I}})$ by marking with {\it black} nodes the elements 
of $I$. The elements of $\Jti$ and their reduced expressions have been obtained using the bijection \eqref{j tilde} and Proposition \ref{red exp} (b).  The Coxeter graph of $(\Wti,\Jti)$ is obtained
from Theorem \ref{gal2}, and the action of the Coxeter generators 
$I$ of $W_I$ by diagram automorphisms of the Coxeter graph of $(\Wti,\Jti)$ may be determined  using Proposition \ref{red exp} (a).  

\medskip

The table contains all possible triples $(W,S,I)$ where $W$ is a finite Coxeter 
group or an affine Weyl group and is irreducible and $I$ is a proper non-empty 
subset of $S$. (For compactness, we include $\widetilde{A_1}$  as $I_2(\infty)$).
 In degenerate cases, that is, for small values of $|S|$, 
the diagram for $(\Wti,\Jti)$ given in the table is not correct, 
but the semidirect product 
decomposition is still correct (see the marks (1), (2) and (3) in the table). 
Here are some detailed explanations:
\begin{itemize}
\itemth{1} If $W$ is of type $\widetilde{B_3}$, then, since 
$D_3=A_3$, we have $\widetilde{D}_3=\widetilde{A}_3$. So the correct Coxeter 
graph of $(\Wti,\Jti)$ is a square of this form
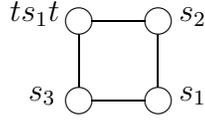
\begin{center}
\begin{picture}(70,50)
\put(20,5){\circle{10}}\put(25,5){\line(1,0){20}}
\put(20,35){\circle{10}}\put(20,30){\line(0,-1){20}}
\put(50,5){\circle{10}}\put(25,35){\line(1,0){20}}
\put(50,35){\circle{10}}\put(50,30){\line(0,-1){20}}
\put(1,5){$s_3$}
\put(-6,35){$ts_1t$}
\put(58,5){$s_1$}
\put(58,35){$s_2$}
\end{picture}
\end{center}

\medskip

\itemth{2} If $W$ is of type $\widetilde{C}_2$, then, since 
$B_2=C_2$, we have $\widetilde{B}_2=\widetilde{C}_2$. So the correct Coxeter 
graph of $(\Wti,\Jti)$ is of the following form 
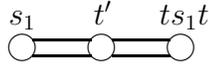
\begin{center}
\begin{picture}(80,25)
\put(10,10){\circle{10}}\put(14,13){\line(1,0){22}}\put(14,7){\line(1,0){22}}
\put(40,10){\circle{10}}\put(44,13){\line(1,0){22}}\put(44,7){\line(1,0){22}}
\put(70,10){\circle{10}}
\put(5,19){$s_1$}
\put(37,19){$t'$}
\put(62,19){$ts_1t$}
\end{picture}
\end{center}

\medskip

\itemth{3} For the diagram marked (3) in the table, there are two values 
of $n$ for which the graph degenerates: if $n=2$, then $D_2= A_1 \times A_1$ 
(this is a standard convention) and so $\widetilde{D}_2=\widetilde{A}_1 \times \widetilde{A}_1$ 
and, if $n=3$, then $D_3=A_3$ so again $\widetilde{D}_3=\widetilde{A}_3$ is a square. We obtain 
the following diagrams:
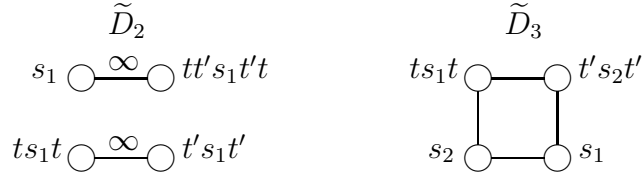
\begin{center}
\begin{picture}(220,70)
\put(30,52){$\widetilde{D}_2$}
\put(20,5){\circle{10}}\put(25,5){\line(1,0){20}}
\put(20,35){\circle{10}}
\put(50,5){\circle{10}}\put(25,35){\line(1,0){20}}
\put(50,35){\circle{10}}
\put(30,8){$\infty$}
\put(30,38){$\infty$}
\put(-6,5){$ts_1t$}
\put(1,35){$s_1$}
\put(58,5){$t's_1t'$}
\put(58,35){$tt's_1t't$}

\put(180,52){$\widetilde{D}_3$}
\put(170,5){\circle{10}}\put(175,5){\line(1,0){20}}
\put(170,35){\circle{10}}\put(170,30){\line(0,-1){20}}
\put(200,5){\circle{10}}\put(175,35){\line(1,0){20}}
\put(200,35){\circle{10}}\put(200,30){\line(0,-1){20}}
\put(151,5){$s_2$}
\put(144,35){$ts_1t$}
\put(208,5){$s_1$}
\put(208,35){$t's_2t'$}
\end{picture}
\end{center}
\end{itemize}

\medskip 
We next explain the notation $t_i$ and $t_i'$ in the 
Coxeter graphs marked (a), (b), (c) and (d) in the table.

\medskip

\begin{itemize}
\itemth{a} Here, $t_1=t$ and $t_{i+1}=s_it_is_i$ ($1 \le i \le n-1$).

\itemth{b} Here, $t_1=t$ and $t_{i+1}=s_it_is_i$ ($1 \le i \le n-1$), 
$t_n'=s_n t_{n-1} s_n$ and $t_i'=s_i t_{i+1}' s_i$ ($1 \le i \le n-1$).

\itemth{c} Here, $t_1=t$ and $t_{i+1}=s_it_is_i$ ($1 \le i \le n-1$), 
$t_n'=t'$ and $t_i'=s_i t_{i+1}' s_i$ ($1 \le i \le n-1$).

\itemth{d} Here, $t_1=t$ and $t_{i+1}=s_it_is_i$ ($1 \le i \le n-1$), 
$t_n'=t' t_n t'$ and $t_i'=s_i t_{i+1}' s_i$ ($1 \le i \le n-1$).
\end{itemize}

\medskip 
Finally, it remains to describe  the $W_I$-action by
automorphisms of $(\Wti,\Jti)$. This may be done by 
describing the automorphism of the Coxeter graph given by the
simple reflections $I$ of $W_I$. Each $s\in I$ acts by conjugation 
on the vertex set $\Jti$ of the Coxeter graph, and in most cases
the action is clear by inspection of the graph. It may be specified by giving
 the induced permutation of the vertex set $\Jti$ of the Coxeter graph.  For example,
 in type $\widetilde{G}_2$ with $I=\set{s_1,s_2}$, the action is given by
 $s_1\mapsto (t,s_1ts_1)$ and $s_2\mapsto (s_1ts_1, s_2s_1ts_1s_2)$
 where the image permutations are written in disjoint cycle notation.
 We will not explicitly list the action in the cases  in which it is obvious by inspection.

\medskip

 The four graphs in the table (or amongst the degenerate graphs discussed above)
 for which the action is perhaps not obvious by inspection are again 
 those designated (a), (b), (c) and (d). For these, the actions  of $W_{I}$ are as follows:
 
 \medskip
 
 \begin{itemize}
\itemth{a} Here, $s_i\mapsto (t_i,t_{i+1})$ for $1\le i\leq n-1$. 

\itemth{b} Here,  $s_i\mapsto (t_i,t_{i+1})(t'_i,t'_{i+1})$ for 
$1\le i\le  n-1$, and 
$s_{n}\mapsto (t_{n-1}, t_n')(t_{n-1}',t_n)$.

\itemth{c} Here, $s_i\mapsto (t_i,t_{i+1})(t_i',t_{i+1}')$ for 
$1\le i\le n-1$.

\itemth{d} Here,  $s_i\mapsto (t_i,t_{i+1})(t'_i,t'_{i+1})$ 
for $1\le i\le n-1$ and $s_n\mapsto (t_n,t_n')$.

\end{itemize}

\medskip
 The resulting permutation representation of $W_I$ is in each case  (a)--(d) isomorphic
 in an obvious way to a standard permutation representation of the classical Weyl group $W_I$ as a group of permutations or signed permutations.

\newpage
\def\espace{\vphantom{\DS{\frac{A}{A}}}}

$$\begin{array}{|c|c|c|c|}
\hline
\espace\text{Type} & \text{Graph of $(W,S,{\boldsymbol{I}})$} & 
\text{Decomposition} & \text{Graph of $(\Wti,\Jti)$} \\
\hline\hline
\espace 
I_2(2m) & 
\begin{picture}(50,20)
\put(0,2){\circle*{6}}\put(-5,7){$\SS{s}$}
\put(3,2){\line(1,0){34}}\put(14,6){$2m$}
\put(40,2){\circle{6}}\put(42,7){$\SS{t}$}
\end{picture}
& (\ZM/2\ZM) \ltimes W(I_2(m)) & 
\begin{picture}(50,20)
\put(0,2){\circle{6}}\put(-10,7){$\SS{sts}$}
\put(3,2){\line(1,0){34}}\put(14,6){$m$}
\put(40,2){\circle{6}}\put(42,7){$\SS{t}$}
\put(-5,10){~}
\end{picture}\\
\hline
\vphantom{\DS{\frac{A^{\DS{A^A}}}{A_{\DS{A_A}}}}} F_4 & 
\begin{picture}(80,20)
\put(0,2){\circle*{6}}\put(-5,9){$\SS{s_2}$}
\put(25,2){\circle*{6}}\put(20,9){$\SS{s_1}$}
\put(50,2){\circle{6}}\put(45,9){$\SS{t_1}$}
\put(75,2){\circle{6}}\put(70,9){$\SS{t_2}$}
\put(3,2){\line(1,0){19}}
\put(27.7,3.2){\line(1,0){19.4}}
\put(27.7,0.8){\line(1,0){19.4}}
\put(53,2){\line(1,0){19}}
\end{picture}
& \SG_3 \ltimes W(D_4) & 
\begin{picture}(70,20)
\put(50,2){\circle{6}}\put(48,9){$\SS{t_2}$}
\put(75,2){\circle{6}}\put(70,9){$\SS{t_1}$}
\put(53,2){\line(1,0){19}}
\put(30,12){\circle{6}}\put(0,12){$\SS{s_1t_1s_1}$}
\put(30,-8){\circle{6}}\put(-15,-8){$\SS{s_2s_1t_1s_1s_2}$}
\put(32.7,10.7){\line(2,-1){14.5}}
\put(32.7,-6.7){\line(2,1){14.5}}
\end{picture}
\\
\hline
\vphantom{\DS{\frac{A^{\DS{A^A}}}{A_{\DS{A_A}}}}}B_n &
\begin{picture}(110,20)
\put(5,2){\circle*{6}}\put(0,9){$\SS{t}$}
\put(30,2){\circle{6}}\put(25,9){$\SS{s_1}$}
\put(55,2){\circle{6}}\put(50,9){$\SS{s_2}$}
\put(105,2){\circle{6}}\put(95,9){$\SS{s_{n-1}}$}
\put(33,2){\line(1,0){19}}
\put(7.7,3.2){\line(1,0){19.4}}
\put(7.7,0.8){\line(1,0){19.4}}
\put(58,2){\line(1,0){12}}
\put(102,2){\line(-1,0){12}}
\put(73,-1){$\cdots$}
\end{picture}
&  (\ZM/2\ZM) \ltimes W(D_n) & 
\begin{picture}(125,20)
\put(50,2){\circle{6}}\put(48,9){$\SS{s_2}$}
\put(75,2){\circle{6}}\put(70,9){$\SS{s_3}$}
\put(53,2){\line(1,0){19}}
\put(30,12){\circle{6}}\put(9,12){$\SS{ts_1t}$}
\put(30,-8){\circle{6}}\put(15,-8){$\SS{s_1}$}
\put(32.7,10.7){\line(2,-1){14.5}}
\put(32.7,-6.7){\line(2,1){14.5}}
\put(120,2){\circle{6}}\put(110,9){$\SS{s_{n-1}}$}
\put(78,2){\line(1,0){11}}
\put(117,2){\line(-1,0){11}}
\put(91,-1){$\cdots$}
\end{picture}
\\
\vphantom{\DS{\frac{A^{\DS{A^A}}}{A_{\DS{A_A}}}}} 
\begin{picture}(30,25)
\put(-2.5,20){$(n \ge 2)$}
\end{picture}&  
\begin{picture}(110,20)
\put(5,2){\circle{6}}\put(0,9){$\SS{t}$}
\put(30,2){\circle*{6}}\put(25,9){$\SS{s_1}$}
\put(55,2){\circle*{6}}\put(50,9){$\SS{s_2}$}
\put(105,2){\circle*{6}}\put(95,9){$\SS{s_{n-1}}$}
\put(33,2){\line(1,0){19}}
\put(7.7,3.2){\line(1,0){19.4}}
\put(7.7,0.8){\line(1,0){19.4}}
\put(58,2){\line(1,0){12}}
\put(102,2){\line(-1,0){12}}
\put(73,-1){$\cdots$}
\end{picture}   & \SG_n \ltimes (\ZM/2\ZM)^n&
\begin{picture}(90,20)
\put(5,2){\circle{6}}\put(0,9){$\SS{t_1}$}
\put(30,2){\circle{6}}\put(25,9){$\SS{t_2}$}
\put(80,2){\circle{6}}\put(76,9){$\SS{t_n}$}
\put(48,-1){$\cdots$}
\put(95,-1){(a)}
\end{picture}
\\
\hline
\vphantom{\DS{\frac{A^{\DS{A^A}}}{A_{\DS{A_A}}}}} \widetilde{G}_2 & 
\begin{picture}(60,20)
\put(5,2){\circle*{6}}\put(0,9){$\SS{t}$}
\put(30,2){\circle{6}}\put(25,9){$\SS{s_1}$}
\put(55,2){\circle{6}}\put(50,9){$\SS{s_2}$}
\put(33,2){\line(1,0){19}}
\put(7.4,3.8){\line(1,0){20.2}}
\put(7.4,0.2){\line(1,0){20.2}}
\put(8,2){\line(1,0){19}}
\end{picture} & (\ZM/2\ZM) \ltimes W(\widetilde{A}_2) & 
\begin{picture}(50,30)
\put(10,2){\circle{6}}\put(-2,2){$\SS{s_1}$}
\put(40,2){\circle{6}}\put(44,2){$\SS{s_2}$}
\put(25,17){\circle{6}}\put(29,20){$\SS{ts_1t}$}
\put(13,2){\line(1,0){24}}
\put(12.3,4.3){\line(1,1){10.6}}
\put(37.7,4.3){\line(-1,1){10.6}}
\end{picture}\\
\vphantom{\DS{\frac{A^{\DS{A^A}}}{A_{\DS{A_A}}}}} &
\begin{picture}(60,20)
\put(5,2){\circle{6}}\put(0,9){$\SS{t}$}
\put(30,2){\circle*{6}}\put(25,9){$\SS{s_1}$}
\put(55,2){\circle*{6}}\put(50,9){$\SS{s_2}$}
\put(33,2){\line(1,0){19}}
\put(7.4,3.8){\line(1,0){20.2}}
\put(7.4,0.2){\line(1,0){20.2}}
\put(8,2){\line(1,0){19}}
\end{picture} & \SG_3 \ltimes W(\widetilde{A}_2)& 
\begin{picture}(50,30)
\put(10,2){\circle{6}}\put(1,2){$\SS{t}$}
\put(40,2){\circle{6}}\put(44,2){$\SS{s_1ts_1}$}
\put(25,17){\circle{6}}\put(29,20){$\SS{s_2s_1ts_1s_2}$}
\put(13,2){\line(1,0){24}}
\put(12.3,4.3){\line(1,1){10.6}}
\put(37.7,4.3){\line(-1,1){10.6}}
\end{picture}\\
\hline
\vphantom{\DS{\frac{A^{\DS{A^A}}}{A_{\DS{A_A}}}}} \widetilde{F}_4 & 
\begin{picture}(100,20)
\put(0,2){\circle*{6}}\put(-5,9){$\SS{s_2}$}
\put(25,2){\circle*{6}}\put(20,9){$\SS{s_1}$}
\put(50,2){\circle{6}}\put(45,9){$\SS{t_1}$}
\put(75,2){\circle{6}}\put(70,9){$\SS{t_2}$}
\put(100,2){\circle{6}}\put(95,9){$\SS{t_3}$}
\put(3,2){\line(1,0){19}}
\put(27.7,3.2){\line(1,0){19.4}}
\put(27.7,0.8){\line(1,0){19.4}}
\put(53,2){\line(1,0){19}}
\put(78,2){\line(1,0){19}}
\end{picture}
& \SG_3 \ltimes W(\widetilde{D}_4) & 
\begin{picture}(85,20)
\put(30,2){\circle{6}}\put(28,9){$\SS{t_2}$}
\put(50,12){\circle{6}}\put(55,12){$\SS{s_1t_1s_1}$}
\put(50,-8){\circle{6}}\put(55,-8){$\SS{s_2s_1t_1s_1s_2}$}
\put(10,12){\circle{6}}\put(-2,12){$\SS{t_1}$}
\put(10,-8){\circle{6}}\put(-2,-8){$\SS{t_3}$}
\put(12.7,10.7){\line(2,-1){14.5}}
\put(12.7,-6.7){\line(2,1){14.5}}
\put(47.3,10.7){\line(-2,-1){14.5}}
\put(47.3,-6.7){\line(-2,1){14.5}}
\end{picture}
\\
\vphantom{\DS{\frac{A^{\DS{A^A}}}{A_{\DS{A_A}}}}} & 
\begin{picture}(100,20)
\put(0,2){\circle{6}}\put(-5,9){$\SS{s_2}$}
\put(25,2){\circle{6}}\put(20,9){$\SS{s_1}$}
\put(50,2){\circle*{6}}\put(45,9){$\SS{t_1}$}
\put(75,2){\circle*{6}}\put(70,9){$\SS{t_2}$}
\put(100,2){\circle*{6}}\put(95,9){$\SS{t_3}$}
\put(3,2){\line(1,0){19}}
\put(27.7,3.2){\line(1,0){19.4}}
\put(27.7,0.8){\line(1,0){19.4}}
\put(53,2){\line(1,0){19}}
\put(78,2){\line(1,0){19}}
\end{picture}
& \SG_4 \ltimes W(\widetilde{D}_4) & 
\begin{picture}(85,20)
\put(30,2){\circle{6}}\put(28,9){$\SS{s_2}$}
\put(50,12){\circle{6}}\put(55,12){$\SS{t_2t_1s_1t_1t_2}$}
\put(50,-8){\circle{6}}\put(55,-8){$\SS{t_3t_2t_1s_1t_1t_2t_3}$}
\put(10,12){\circle{6}}\put(-18,12){$\SS{t_1s_1t_1}$}
\put(10,-8){\circle{6}}\put(-2,-8){$\SS{s_1}$}
\put(12.7,10.7){\line(2,-1){14.5}}
\put(12.7,-6.7){\line(2,1){14.5}}
\put(47.3,10.7){\line(-2,-1){14.5}}
\put(47.3,-6.7){\line(-2,1){14.5}}
\end{picture}
\\
\hline
\vphantom{\DS{\frac{A^{\DS{A^A}}}{A_{\DS{A_A}}}}}\widetilde{B}_n &
\begin{picture}(125,20)
\put(5,2){\circle*{6}}\put(0,9){$\SS{t}$}
\put(30,2){\circle{6}}\put(25,9){$\SS{s_1}$}
\put(80,2){\circle{6}}\put(70,9){$\SS{s_{n-2}}$}
\put(100,12){\circle{6}}\put(105,12){$\SS{s_{n-1}}$}
\put(100,-8){\circle{6}}\put(105,-8){$\SS{s_{n}}$}
\put(82.6,3.5){\line(2,1){14.7}}
\put(82.6,0.5){\line(2,-1){14.7}}
\put(7.7,3.2){\line(1,0){19.4}}
\put(7.7,0.8){\line(1,0){19.4}}
\put(33,2){\line(1,0){12}}
\put(77,2){\line(-1,0){12}}
\put(48,-1){$\cdots$}
\end{picture}
&  (\ZM/2\ZM) \ltimes W(\widetilde{D}_n) & 
\begin{picture}(130,20)
\put(40,2){\circle{6}}\put(38,9){$\SS{s_2}$}
\put(20,12){\circle{6}}\put(1,12){$\SS{ts_1t}$}
\put(20,-8){\circle{6}}\put(5,-8){$\SS{s_1}$}
\put(22.7,10.7){\line(2,-1){14.5}}
\put(22.7,-6.7){\line(2,1){14.5}}
\put(85,2){\circle{6}}\put(75,9){$\SS{s_{n-2}}$}
\put(105,12){\circle{6}}\put(110,12){$\SS{s_{n-1}}$}
\put(105,-8){\circle{6}}\put(110,-8){$\SS{s_{n}}$}
\put(87.6,3.5){\line(2,1){14.7}}
\put(87.6,0.5){\line(2,-1){14.7}}
\put(43,2){\line(1,0){11}}
\put(82,2){\line(-1,0){11}}
\put(56,-1){$\cdots$}
\put(56,-12){(1)}
\end{picture}
\\
\begin{picture}(30,45)
\put(-2.5,35){$(n \ge 3)$}
\end{picture}
&\begin{picture}(125,45)
\put(5,17){\circle{6}}\put(0,24){$\SS{t}$}
\put(30,17){\circle*{6}}\put(25,24){$\SS{s_1}$}
\put(80,17){\circle*{6}}\put(70,24){$\SS{s_{n-2}}$}
\put(100,27){\circle*{6}}\put(105,27){$\SS{s_{n-1}}$}
\put(100,7){\circle*{6}}\put(105,7){$\SS{s_{n}}$}
\put(82.6,18.5){\line(2,1){14.7}}
\put(82.6,15.5){\line(2,-1){14.7}}
\put(7.7,18.2){\line(1,0){19.4}}
\put(7.7,15.8){\line(1,0){19.4}}
\put(33,17){\line(1,0){12}}
\put(77,17){\line(-1,0){12}}
\put(48,14){$\cdots$}
\end{picture}
&  
\begin{picture}(100,45)
\put(0,14){$W(D_n) \ltimes \bigl(W(\widetilde{A}_1)\bigr)^n$}
\end{picture}&
\begin{picture}(90,45)
\put(5,27){\circle{6}}\put(0,34){$\SS{t_1}$}\put(5,24){\line(0,-1){14}}\put(-5,15){$\SS{\infty}$}
\put(30,27){\circle{6}}\put(25,34){$\SS{t_2}$}\put(30,24){\line(0,-1){14}}
\put(20,15){$\SS{\infty}$}
\put(80,27){\circle{6}}\put(76,34){$\SS{t_n}$}\put(80,24){\line(0,-1){14}}
\put(70,15){$\SS{\infty}$}
\put(48,14){$\cdots$}
\put(5,7){\circle{6}}\put(0,-4){$\SS{t_1'}$}
\put(30,7){\circle{6}}\put(25,-4){$\SS{t_2'}$}
\put(80,7){\circle{6}}\put(76,-4){$\SS{t_n'}$}
\put(95,15){(b)}
\end{picture}\vphantom{\DS{\frac{A^{\DS{A^A}}}{A_{\DS{A_A}}}}}
\\
\hline
\vphantom{\DS{\frac{A^{\DS{A^A}}}{A_{\DS{A_A}}}}}\widetilde{C}_n &
\begin{picture}(110,20)
\put(5,2){\circle*{6}}\put(0,9){$\SS{t}$}
\put(30,2){\circle{6}}\put(25,9){$\SS{s_1}$}
\put(80,2){\circle{6}}\put(70,9){$\SS{s_{n-1}}$}
\put(105,2){\circle{6}}\put(103,9){$\SS{t'}$}
\put(7.7,3.2){\line(1,0){19.4}}
\put(7.7,0.8){\line(1,0){19.4}}
\put(82.7,3.2){\line(1,0){19.4}}
\put(82.7,0.8){\line(1,0){19.4}}
\put(33,2){\line(1,0){12}}
\put(77,2){\line(-1,0){12}}
\put(48,-1){$\cdots$}
\end{picture}
&  (\ZM/2\ZM) \ltimes W(\widetilde{B}_n) & 
\begin{picture}(130,20)
\put(40,2){\circle{6}}\put(38,9){$\SS{s_2}$}
\put(20,12){\circle{6}}\put(1,12){$\SS{ts_1t}$}
\put(20,-8){\circle{6}}\put(5,-8){$\SS{s_1}$}
\put(22.7,10.7){\line(2,-1){14.5}}
\put(22.7,-6.7){\line(2,1){14.5}}
\put(85,2){\circle{6}}\put(75,9){$\SS{s_{n-1}}$}
\put(110,2){\circle{6}}\put(108,9){$\SS{t'}$}
\put(87.7,3.2){\line(1,0){19.4}}
\put(87.7,0.8){\line(1,0){19.4}}
\put(43,2){\line(1,0){11}}
\put(82,2){\line(-1,0){11}}
\put(56,-1){$\cdots$}
\put(56,-12){(2)}
\end{picture}
\\
\begin{picture}(30,45)
\put(-2.5,35){($n \ge 2$)}
\end{picture}
\vphantom{\DS{\frac{A^{\DS{A^A}}}{A_{\DS{A_A}}}}}
 &
\begin{picture}(110,45)
\put(5,17){\circle{6}}\put(0,24){$\SS{t}$}
\put(30,17){\circle*{6}}\put(25,24){$\SS{s_1}$}
\put(80,17){\circle*{6}}\put(70,24){$\SS{s_{n-1}}$}
\put(105,17){\circle{6}}\put(103,24){$\SS{t'}$}
\put(7.7,18.2){\line(1,0){19.4}}
\put(7.7,15.8){\line(1,0){19.4}}
\put(82.7,18.2){\line(1,0){19.4}}
\put(82.7,15.8){\line(1,0){19.4}}
\put(33,17){\line(1,0){12}}
\put(77,17){\line(-1,0){12}}
\put(48,14){$\cdots$}
\end{picture}
&\begin{picture}(83,45)
\put(0,14){$\SG_n \ltimes \bigl(W(\widetilde{A}_1)\bigr)^n$}
\end{picture}&
\begin{picture}(90,45)
\put(5,27){\circle{6}}\put(0,34){$\SS{t_1}$}\put(5,24){\line(0,-1){14}}\put(-5,15){$\SS{\infty}$}
\put(30,27){\circle{6}}\put(25,34){$\SS{t_2}$}\put(30,24){\line(0,-1){14}}
\put(20,15){$\SS{\infty}$}
\put(80,27){\circle{6}}\put(76,34){$\SS{t_n}$}\put(80,24){\line(0,-1){14}}
\put(70,15){$\SS{\infty}$}
\put(48,14){$\cdots$}
\put(5,7){\circle{6}}\put(0,-4){$\SS{t_1'}$}
\put(30,7){\circle{6}}\put(25,-4){$\SS{t_2'}$}
\put(80,7){\circle{6}}\put(76,-4){$\SS{t_n'}$}
\put(95,15){(c)}
\end{picture}\vphantom{\DS{\frac{A^{\DS{A^A}}}{A_{\DS{A_A}}}}}
\\
\vphantom{\DS{\frac{A^{\DS{A^A}}}{A_{\DS{A_A}}}}}
 &
\begin{picture}(110,45)
\put(5,17){\circle{6}}\put(0,24){$\SS{t}$}
\put(30,17){\circle*{6}}\put(25,24){$\SS{s_1}$}
\put(80,17){\circle*{6}}\put(70,24){$\SS{s_{n-1}}$}
\put(105,17){\circle*{6}}\put(103,24){$\SS{t'}$}
\put(7.7,18.2){\line(1,0){19.4}}
\put(7.7,15.8){\line(1,0){19.4}}
\put(82.7,18.2){\line(1,0){19.4}}
\put(82.7,15.8){\line(1,0){19.4}}
\put(33,17){\line(1,0){12}}
\put(77,17){\line(-1,0){12}}
\put(48,14){$\cdots$}
\end{picture}
&\begin{picture}(100,45)
\put(0,14){$W(B_n) \ltimes \bigl(W(\widetilde{A}_1)\bigr)^n$}
\end{picture}&
\begin{picture}(90,45)
\put(5,27){\circle{6}}\put(0,34){$\SS{t_1}$}\put(5,24){\line(0,-1){14}}\put(-5,15){$\SS{\infty}$}
\put(30,27){\circle{6}}\put(25,34){$\SS{t_2}$}\put(30,24){\line(0,-1){14}}
\put(20,15){$\SS{\infty}$}
\put(80,27){\circle{6}}\put(76,34){$\SS{t_n}$}\put(80,24){\line(0,-1){14}}
\put(70,15){$\SS{\infty}$}
\put(48,14){$\cdots$}
\put(5,7){\circle{6}}\put(0,-4){$\SS{t_1'}$}
\put(30,7){\circle{6}}\put(25,-4){$\SS{t_2'}$}
\put(80,7){\circle{6}}\put(76,-4){$\SS{t_n'}$}\put(95,15){(d)}
\end{picture}\vphantom{\DS{\frac{A^{\DS{A^A}}}{A_{\DS{A_A}}}}}
\\
\vphantom{\DS{\frac{A^{\DS{A^A}}}{A_{\DS{A_A}}}}}
 &
\begin{picture}(110,20)
\put(5,2){\circle*{6}}\put(0,9){$\SS{t}$}
\put(30,2){\circle{6}}\put(25,9){$\SS{s_1}$}
\put(80,2){\circle{6}}\put(70,9){$\SS{s_{n-1}}$}
\put(105,2){\circle*{6}}\put(103,9){$\SS{t'}$}
\put(7.3,3.2){\line(1,0){19.9}}
\put(7.3,0.8){\line(1,0){19.9}}
\put(82.7,3.2){\line(1,0){20}}
\put(82.7,0.8){\line(1,0){20}}
\put(33,2){\line(1,0){12}}
\put(77,2){\line(-1,0){12}}
\put(48,-1){$\cdots$}
\end{picture}

& (\SG_2 \times \SG_2) \ltimes W(\widetilde{D}_n)&
\begin{picture}(140,20)
\put(40,2){\circle{6}}\put(38,9){$\SS{s_2}$}
\put(20,12){\circle{6}}\put(1,12){$\SS{ts_1t}$}
\put(20,-8){\circle{6}}\put(5,-8){$\SS{s_1}$}
\put(22.7,10.7){\line(2,-1){14.5}}
\put(22.7,-6.7){\line(2,1){14.5}}
\put(85,2){\circle{6}}\put(75,9){$\SS{s_{n-2}}$}
\put(105,12){\circle{6}}\put(110,12){$\SS{s_{n-1}}$}
\put(105,-8){\circle{6}}\put(110,-8){$\SS{t's_{n-1}t'}$}
\put(87.6,3.5){\line(2,1){14.7}}
\put(87.6,0.5){\line(2,-1){14.7}}
\put(43,2){\line(1,0){11}}
\put(82,2){\line(-1,0){11}}
\put(56,-1){$\cdots$}
\put(56,-11){(3)}
\end{picture}\\
\hline
\end{array}$$


\newpage

\begin{thebibliography}{131}
\bibitem{BBR} {\sc D. Bessis, C. Bonnaf\'e \& R. Rouquier}, 
Quotients et extensions de groupes de r\'eflexion, {\it Math. Ann.} 
{\bf 323} (2002), 405-436.

\bibitem{semicontinu} {\sc C. Bonnaf\'e}, Semi-continuit\'e des cellules 
de Kazhdan-Lusztig, preprint (2008).

\bibitem{BP} {\sc C. Bonnaf\'e \& G. Pfeiffer}, 
Around Solomon's descent algebras, 
to appear in {\it  Algebras and Representation Theory}.

\bibitem{bourbaki} {\sc N. Bourbaki}, {\it Groupes et alg\`ebres de Lie, 
chapitres IV, V et VI}, Hermann, Paris, 1968.

\bibitem{brink} {\sc B. Brink}, The set of dominance-minimal roots, 
{\it J. Algebra} {\bf 206} (1998), 371--412.

\bibitem{dyer} {\sc M. Dyer}, Reflection subgroups of Coxeter groups, 
{\it J. Algebra} {\bf 135} (1990), 57--73.

\bibitem{gal} {\sc S. Gal}, On normal subgroups of Coxeter groups generated by standard parabolic subgroups, {\it Geom. Ded.} {\bf 115} (2005), 65--78.

\bibitem{geck}
{\sc M. Geck and G. Pfeiffer}, {\em Characters of finite Coxeter groups and 
Iwahori--Hecke algebras}, London Math. Soc. Monographs,
New Series {\bf 21}, Oxford University Press, 2000.

\bibitem{solomon} {\sc L. Solomon}, 
A Mackey formula in the group ring of a Coxeter group, 
{\it J. Algebra} {\bf 41} (1976), 255-264.
\end{thebibliography}
\end{document}